\documentclass{llncs}

\usepackage[T1]{fontenc}
\usepackage[utf8]{inputenc}

\usepackage{amssymb, amsmath,graphicx,graphics,enumerate,tkz-graph}
\usepackage{microtype}
\usepackage[hypertexnames=false]{hyperref}

\title{Clique-width and Well-Quasi-Ordering of Triangle-Free Graph Classes\thanks{This research was supported by EPSRC (EP/K025090/1 and EP/L020408/1) and the Leverhulme Trust RPG-2016-258.
An extended abstract of this paper appeared in the proceedings of WG 2017~\cite{DLP}.}}
 
\author{Konrad K. Dabrowski\inst{1} \and Vadim V. Lozin\inst{2} \and Dani\"el~Paulusma\inst{1}}
\institute{
Department of Computer Science, Durham University,\\
Lower Mountjoy, South Road, Durham DH1 3LE, United Kingdom\\
\texttt{\{konrad.dabrowski,daniel.paulusma\}@durham.ac.uk}
\and
Mathematics Institute, University of Warwick,\\
Coventry CV4 7AL, United Kingdom\\
\texttt{v.lozin@warwick.ac.uk}}

\DeclareMathOperator{\cw}{cw}

\newtheorem{oproblem}{Open Problem}

\newcommand{\NP}{{\sf NP}}
\newcommand{\ssi}{\subseteq_i}
\newcommand{\si}{\supseteq_i}

\newcounter{ctrclaim}[theorem]

\newcommand{\clm}[1]{\medskip\phantomsection\refstepcounter{ctrclaim}\noindent{\em Claim \thectrclaim. }{\em #1}\\}
\newcommand{\clmnonewline}[1]{\medskip\phantomsection\refstepcounter{ctrclaim}\noindent{\em Claim \thectrclaim. }{\em #1}}

\begin{document}
\maketitle
\begin{sloppypar}
\begin{abstract}
Daligault, Rao and Thomass{\'e} asked whether every hereditary graph class that is well-quasi-ordered by the induced subgraph relation has bounded clique-width.
Lozin, Razgon and Zamaraev (JCTB 2017+) gave a negative answer to this question, but their counterexample is a class that can only be characterised by infinitely many forbidden induced subgraphs.
This raises the issue of whether the question has a positive answer for finitely defined hereditary graph classes.
Apart from two stubborn cases, this has been confirmed when at most two induced subgraphs $H_1,H_2$ are forbidden.
We confirm it for one of the two stubborn cases, namely for the $(H_1,H_2)=(\mbox{triangle},P_2+\nobreak P_4)$ case, by proving that the class of $(\mbox{triangle},P_2+\nobreak P_4)$-free graphs has bounded clique-width and is well-quasi-ordered.
Our technique is based on a special decomposition of $3$-partite graphs.
We also use this technique to prove that the class of $(\mbox{triangle},P_1+\nobreak P_5)$-free graphs, which is known to have bounded clique-width, is well-quasi-ordered. 
Our results enable us to complete the classification of graphs~$H$ for which the class of $(\mbox{triangle},H)$-free graphs is
well-quasi-ordered.
\end{abstract}
\end{sloppypar}

\section{Introduction}\label{s-intro}

A graph class~${\cal G}$ is well-quasi-ordered by some containment relation if for any infinite sequence $G_0, G_1, \ldots$ of graphs in~${\cal G}$, there is a pair $i,j$ with $i<j$ such that~$G_i$ is contained in~$G_j$.
A graph class~${\cal G}$ has bounded clique-width if there exists a constant~$c$ such that every graph in~${\cal G}$ has clique-width at most~$c$.
Both being well-quasi-ordered and having bounded clique-width are highly desirable properties of graph classes in the areas of discrete mathematics and theoretical computer science.
To illustrate this, let us mention the seminal project of Robertson and Seymour on graph minors that culminated in 2004 in the proof of Wagner's conjecture, which states that the set of all finite graphs is well-quasi-ordered by the minor relation.
As an algorithmic consequence, given a minor-closed graph class, it is possible to test in cubic time whether a given graph belongs to this class.
The algorithmic importance of having bounded clique-width follows from the fact that many well-known \NP-hard problems, such as {\sc Graph Colouring} and {\sc Hamilton Cycle}, become polynomial-time solvable for graph classes of bounded clique-width (this follows from combining results from several papers~\cite{CMR00,EGW01,KR03b,Ra07} with a result of Oum and Seymour~\cite{OS06}).

Courcelle~\cite{Co14} proved that the class of graphs obtained from graphs of clique-width~$3$ via one or more edge contractions has unbounded clique-width. Hence the clique-width of a graph can be much smaller than the clique-width of its minors.
On the other hand, the clique-width of a graph is at least the clique-width of any of its induced subgraphs (see, for example,~\cite{CO00}). We therefore focus on {\em hereditary} classes, that is, on graph classes that are closed under taking induced subgraphs.
It is readily seen that a class of graphs is hereditary if and only if it can be characterised by a unique set~${\cal F}$ of minimal forbidden induced subgraphs.
Our underlying research goal is to increase our understanding of the relation between well-quasi-orders and clique-width of hereditary graph classes.

As a start, we note that the hereditary class of graphs of degree at most~$2$ is not well-quasi-ordered by the induced subgraph relation, as it contains the class of cycles, which form an infinite antichain.
As every graph of degree at most~$2$ has clique-width at most~$4$, having bounded clique-width does not imply being well-quasi-ordered by the induced subgraph relation.
In 2010, Daligault, Rao and Thomass{\'e}~\cite{DRT10} asked about the reverse implication: 
does every hereditary graph class that is well-quasi-ordered by the induced subgraph relation have bounded clique-width? 
In 2015, Lozin, Razgon and Zamaraev~\cite{LRZ15} gave a negative answer.
As the set~${\cal F}$ of minimal forbidden induced subgraphs in their counter-example is infinite, the question of Daligault, Rao and Thomass{\'e}~\cite{DRT10} remains open for {\em finitely defined} hereditary graph classes, that is, hereditary graph classes for which~${\cal F}$ is finite.

\begin{conjecture}[\cite{LRZ15}]\label{c-f}
If a finitely defined hereditary class of graphs~${\cal G}$ is well-quasi-ordered by the induced subgraph relation, then~${\cal G}$ has bounded clique-width.
\end{conjecture} 

\noindent
If Conjecture~\ref{c-f} is true, then for finitely defined hereditary graph classes the aforementioned algorithmic consequences of having bounded clique-width also hold for the property of being well-quasi-ordered by the induced subgraph relation.
A hereditary graph class defined by a single forbidden induced subgraph~$H$ is well-quasi-ordered by the induced subgraph relation if and only if it has bounded clique-width  if and only if~$H$ is an induced subgraph of~$P_4$ (see, for instance,~\cite{DP15,Da90,KL11}).
Hence Conjecture~\ref{c-f} holds when~${\cal F}$ has size~$1$.
We consider the case when~${\cal F}$ has size~$2$, say ${\cal F}=\{H_1,H_2\}$.
Such graph classes are said to be {\em bigenic} or {\em $(H_1,H_2)$-free} graph classes.
In this case Conjecture~\ref{c-f} is also known to be true except for two stubborn open cases, namely $(H_1,H_2)=(K_3,P_2+\nobreak P_4)$ and $(H_1,H_2)=(\overline{P_1+P_4},P_2+\nobreak P_3)$; see~\cite{DLP16}.

\begin{sloppypar}
\medskip
\noindent
{\bf Our Results.} In Section~\ref{s-main}, we prove that 
Conjecture~\ref{c-f} holds for the class of $(K_3,P_2+\nobreak P_4)$-free graphs 
by showing that the class of $(K_3,P_2+\nobreak P_4)$-free graphs 
has bounded clique-width and is well-quasi-ordered by the induced subgraph relation.
We do this by using a general technique explained in Section~\ref{sec:partitioning}.
This technique is based on extending (a labelled version of) well-quasi-orderability or boundedness of clique-width of the bipartite graphs in a hereditary graph class~$X$ to a special subclass of $3$-partite graphs in~$X$.
The crucial property of these $3$-partite graphs is that no three vertices from the three different partition classes form a clique or independent set.
We say that such $3$-partite graphs {\em curious}.
A more restricted version of this concept was used to prove that $(K_3,P_1+\nobreak P_5)$-free graphs have bounded clique-width~\cite{DDP15}.
In Section~\ref{s-main} we show how to generalise results for curious  $(K_3,P_2+\nobreak P_4)$-free graphs to the whole class of  $(K_3,P_2+\nobreak P_4)$-free graphs.
In the same section we also show how to apply our technique to prove that the class of $(K_3,P_1+\nobreak P_5)$-free graphs is well-quasi-ordered
(it was already known~\cite{DDP15} that the class of $(K_3,P_1+\nobreak P_5)$-free graphs has bounded clique-width).
We note that our results also imply that the class of $(K_3,P_1+\nobreak 2P_2)$-free graphs is well-quasi-ordered, which was not previously known (see~\cite{AL15,KL11}).
See also \figurename~\ref{fig:forb} for pictures of the forbidden induced subgraphs mentioned in this paragraph.
\end{sloppypar}

\begin{figure}
\begin{center}
\begin{tabular}{cccc}
\begin{minipage}{0.24\textwidth}
\begin{center}
\begin{tikzpicture}
\coordinate (x1) at (90:1) ;
\coordinate (x2) at (90+120:1) ;
\coordinate (x3) at (90+240:1) ;
\draw [fill=black] (x1) circle (2pt) ;
\draw [fill=black] (x2) circle (2pt) ;
\draw [fill=black] (x3) circle (2pt) ;
\draw [thick] (x1) -- (x2) -- (x3) -- (x1);
\end{tikzpicture}
\end{center}
\end{minipage}
&
\begin{minipage}{0.24\textwidth}
\begin{center}
\begin{tikzpicture}
\coordinate (x1) at (90:1) ;
\coordinate (x2) at (90+72:1) ;
\coordinate (x3) at (90+144:1) ;
\coordinate (x4) at (90+216:1) ;
\coordinate (x5) at (90+288:1) ;
\draw [fill=black] (x1) circle (2pt) ;
\draw [fill=black] (x2) circle (2pt) ;
\draw [fill=black] (x3) circle (2pt) ;
\draw [fill=black] (x4) circle (2pt) ;
\draw [fill=black] (x5) circle (2pt) ;
\draw [thick] (x2) -- (x3);
\draw [thick] (x4) -- (x5);
\end{tikzpicture}
\end{center}
\end{minipage}
&
\begin{minipage}{0.24\textwidth}
\begin{center}
\begin{tikzpicture}
\coordinate (x1) at (0:1) ;
\coordinate (x2) at (60:1) ;
\coordinate (x3) at (120:1) ;
\coordinate (x4) at (180:1) ;
\coordinate (x5) at (240:1) ;
\coordinate (x6) at (300:1) ;
\draw [fill=black] (x1) circle (2pt) ;
\draw [fill=black] (x2) circle (2pt) ;
\draw [fill=black] (x3) circle (2pt) ;
\draw [fill=black] (x4) circle (2pt) ;
\draw [fill=black] (x5) circle (2pt) ;
\draw [fill=black] (x6) circle (2pt) ;
\draw [thick] (x5) -- (x6) -- (x1) -- (x2) -- (x3);
\end{tikzpicture}
\end{center}
\end{minipage}
&
\begin{minipage}{0.24\textwidth}
\begin{center}
\begin{tikzpicture}
\coordinate (x1) at (0:1) ;
\coordinate (x2) at (60:1) ;
\coordinate (x3) at (120:1) ;
\coordinate (x4) at (180:1) ;
\coordinate (x5) at (240:1) ;
\coordinate (x6) at (300:1) ;
\draw [fill=black] (x1) circle (2pt) ;
\draw [fill=black] (x2) circle (2pt) ;
\draw [fill=black] (x3) circle (2pt) ;
\draw [fill=black] (x4) circle (2pt) ;
\draw [fill=black] (x5) circle (2pt) ;
\draw [fill=black] (x6) circle (2pt) ;
\draw [thick] (x2) -- (x3);
\draw [thick] (x4) -- (x5) -- (x6) -- (x1);
\end{tikzpicture}
\end{center}
\end{minipage}\\
\\
$K_3$ & $P_1+\nobreak 2P_2$ & $P_1+\nobreak P_5$ & $P_2+\nobreak P_4$
\end{tabular}
\end{center}
\caption{\label{fig:forb}The forbidden induced subgraphs considered in our results.}
\end{figure}
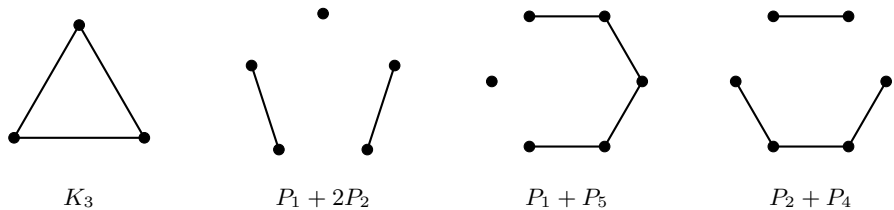

\begin{sloppypar}
\medskip
\noindent
{\bf Dichotomies.}
Previously, well-quasi-orderability was known for $(K_3,P_6)$-free graphs~\cite{AL15}, $(P_2+\nobreak P_4)$-free bipartite graphs~\cite{KV11} and $(P_1+\nobreak P_5)$-free bipartite graphs~\cite{KV11}.
It has also been shown that $H$-free bipartite graphs are not well-quasi-ordered if~$H$ contains an induced $3P_1+\nobreak P_2$~\cite{KL11}, $3P_2$~\cite{Di92} or $2P_3$~\cite{KV11}.
Moreover, for every $s\geq 1$, the class of $(K_3,sP_1)$-free graphs is finite due to Ramsey's Theorem~\cite{Ra30}.
The above results lead to the following known dichotomy for $H$-free bipartite graphs.
\end{sloppypar}

\begin{theorem}\label{t-main0}
Let~$H$ be a graph.
The class of $H$-free bipartite graphs is well-quasi-ordered by the induced subgraph relation if and only if $H=sP_1$ for some $s\geq 1$ or~$H$ is an induced subgraph of $P_1+\nobreak P_5$, $P_2+\nobreak P_4$ or~$P_6$.
\end{theorem}
Now combining the aforementioned known results for $(K_3,H)$-free graphs and $H$-free bipartite graphs with our new results yields the following new dichotomy for $H$-free triangle-free graphs, which is exactly the same as the one in Theorem~\ref{t-main0}.

\begin{theorem}\label{t-main}
Let~$H$ be a graph. The class of $(K_3,H)$-free graphs is well-quasi-ordered by the induced subgraph relation if and only if $H=sP_1$ for some $s\geq 1$ or $H$ is an induced subgraph of $P_1+\nobreak P_5$, $P_2+\nobreak P_4$, or~$P_6$.
\end{theorem}
Besides our technique based on curious graphs, we also expect that Theorem~\ref{t-main} will itself be a useful ingredient for showing results for other graph classes, just as Theorem~\ref{t-main0} has already proven to be useful (see e.g.~\cite{KV11}).

For clique-width the following dichotomy is known for $H$-free bipartite graphs. 

\begin{theorem}[\cite{DP14}]\label{t-bipartite}
Let~$H$ be a graph.
The class of $H$-free bipartite graphs has bounded
clique-width if and only~if  $H=sP_1$ for some $s\geq 1$ or~$H$ is an induced subgraph of $K_{1,3}+3P_1$, 
$K_{1,3}+\nobreak P_2$, $P_1+\nobreak S_{1,1,3}$ or~$S_{1,2,3}$.
\end{theorem}
It would be interesting to determine whether $(K_3,H)$-free graphs allow the same dichotomy with respect to the boundedness of their clique-width.
The evidence so far is affirmative, but in order to answer this question two remaining cases need to be solved, namely $(H_1,H_2)=(K_3,P_1+\nobreak S_{1,1,3})$
and $(H_1,H_2)=(K_3,S_{1,2,3})$; see Section~\ref{sec:prelim} for the definition of the graph~$S_{h,i,j}$. Both cases turn out to be highly non-trivial; in particular, the class of $(K_3,P_1+\nobreak S_{1,1,3})$-free graphs contains the class of $(K_3,P_1+\nobreak P_5)$-free graphs, and the class of $(K_3,S_{1,2,3})$-free graphs contains both the classes of $(K_3,P_1+\nobreak P_5)$-free and $(K_3,P_2+\nobreak P_4)$-free graphs.

In Section~\ref{s-con} we give state-of-the-art summaries for well-quasi-orderability and boundedness of clique-width of bigenic graph classes (which include our new results), together with an overview of the missing cases for both problems (which include the missing cases mentioned in this section).

\section{Preliminaries}\label{sec:prelim}

Throughout the paper, we consider only finite, undirected graphs without multiple edges or self-loops.
Below, we define further graph terminology.

The {\em disjoint union} $(V(G)\cup V(H), E(G)\cup E(H))$ of two vertex-disjoint graphs~$G$ and~$H$ is denoted by~$G+\nobreak H$ and the disjoint union of~$r$ copies of a graph~$G$ is denoted by~$rG$. The {\em complement}~$\overline{G}$ of a graph~$G$ has vertex set $V(\overline{G})=\nobreak V(G)$ and an edge between two distinct vertices $u,v$
if and only if $uv\notin E(G)$. 
For a subset $S\subseteq V(G)$, we let~$G[S]$ denote the subgraph of~$G$ {\em induced} by~$S$, which has vertex set~$S$ and edge set $\{uv\; |\; u,v\in S, uv\in E(G)\}$.
If $S=\{s_1,\ldots,s_r\}$ then, to simplify notation, we may also write $G[s_1,\ldots,s_r]$ instead of $G[\{s_1,\ldots,s_r\}]$.
We use $G \setminus S$ to denote the graph obtained from~$G$ by deleting every vertex in~$S$, that is, $G \setminus S = G[V(G)\setminus S]$.
We write $G'\subseteq_i G$ to indicate that~$G'$ is an induced subgraph of~$G$.

The graphs $C_r$, $K_r$, $K_{1,r-1}$ and~$P_r$ denote the cycle, complete graph, star and path on~$r$ vertices, respectively.
The graphs~$K_3$ and~$K_{1,3}$ are also called the {\em triangle} and {\em claw}, respectively.
A graph~$G$ is a {\em linear forest} if every component of~$G$ is a path (on at least one vertex).
The graph~$S_{h,i,j}$, for $1\leq h\leq i\leq j$, denotes the {\em subdivided claw}, that is,
the tree that has only one vertex~$x$ of degree~$3$ and exactly three leaves, which are of distance~$h$,~$i$ and~$j$ from~$x$, respectively.
Observe that $S_{1,1,1}=K_{1,3}$.
We let~${\cal S}$ denote the class of graphs, each connected component of which is either a subdivided claw or a path.

For a set of graphs $\{H_1,\ldots,H_p\}$, a graph~$G$ is {\em $(H_1,\ldots,H_p)$-free} if it has no induced subgraph isomorphic to a graph in $\{H_1,\ldots,H_p\}$;
if~$p=1$, we may write $H_1$-free instead of $(H_1)$-free.
A graph is~{\em $k$-partite} if its vertex can be partitioned into~$k$ (possibly empty) independent sets; $2$-partite graphs are also known as {\em bipartite} graphs.
The {\em biclique} or {\em complete bipartite graph}~$K_{r,s}$ is the bipartite graph with sets in the partition of size~$r$ and~$s$ respectively, such that every vertex in one set is adjacent to every vertex in the other set.
For a graph $G=(V,E)$, the set $N(u)=\{v\in V\; |\; uv\in E\}$ denotes the {\em neighbourhood} of $u\in V$.
Let~$X$ be a set of vertices in~$G$.
A vertex $y\in V\setminus X$ is {\em complete} to~$X$ if it is adjacent to every vertex of~$X$ and {\em anti-complete} to~$X$ if it is {\em non}-adjacent to every vertex of~$X$.
A set of vertices $Y\subseteq V\setminus X$ is {\em complete} (resp. {\em anti-complete}) to~$X$ if every vertex in~$Y$ is complete (resp. anti-complete) to~$X$.
If~$X$ and~$Y$ are disjoint sets of vertices in a graph, we say that the edges between these two sets form a {\em matching} if each vertex in~$X$ has at most one neighbour in~$Y$ and vice versa
(if each vertex has exactly one such neighbour, we say that the matching is {\em perfect}).
Similarly, the edges between these sets form a {\em co-matching} if each vertex in~$X$ has at most one non-neighbour in~$Y$ and vice versa.
We say that the set~$X$ {\em dominates}~$Y$ if every vertex of~$Y$ has at least one neighbour in~$X$.
Similarly, a vertex~$x$ {\em dominates}~$Y$ if every vertex of~$Y$ is adjacent to~$x$.
A vertex $y\in V\setminus X$ {\em distinguishes}~$X$ if~$y$ has both a neighbour and a non-neighbour in~$X$. 
The set~$X$ is a {\em module} of~$G$ if no vertex in $V\setminus X$ distinguishes~$X$.
A module~$X$ is {\em non-trivial} if $1<|X|<|V|$, otherwise it is {\em trivial}. 
A graph is {\em prime} if it has only trivial modules.
Two vertices are {\em false twins} if they have the same neighbourhood (note that such vertices must be non-adjacent).
Clearly any prime graph on at least three vertices cannot contain a pair of false twins, as any such pair of vertices would form a non-trivial module.

We will use the following structural result.

\begin{lemma}[\cite{DDP15}]\label{lem:noC5-partial}
Let~$G$ be a connected $(K_3,C_5,S_{1,2,3})$-free graph that does not contain a pair of false twins.
Then~$G$ is either bipartite or a cycle.
\end{lemma}

\subsection{Clique-width}\label{sec:clique-width}
The {\em clique-width}~$\cw(G)$ of a graph~$G$ is the minimum
number of labels needed to
construct~$G$ by
using the following four operations:
\begin{enumerate}
\item $i(v)$: creating a new graph consisting of a single vertex~$v$ with label~$i$;
\item $G_1\oplus \nobreak G_2$: taking the disjoint union of two labelled graphs~$G_1$ and~$G_2$;
\item $\eta_{i,j}$: joining each vertex with label~$i$ to each vertex with label~$j$ ($i\neq j$);
\item $\rho_{i\rightarrow j}$: renaming label~$i$ to~$j$.
\end{enumerate}
\noindent
A class of graphs~${\cal G}$ has {\em bounded} clique-width if there is a constant~$c$ such that the clique-width of every graph in~${\cal G}$ is at most~$c$; otherwise the clique-width  is {\em unbounded}.

Let~$G$ be a graph.
We define the following operations.
For an induced subgraph $G'\ssi G$, the {\em subgraph complementation} operation (acting on~$G$ with respect to~$G'$) replaces every edge present in~$G'$
by a non-edge, and vice versa.
Similarly, for two disjoint vertex subsets~$S$ and~$T$ in~$G$, the {\em bipartite complementation} operation with respect to~$S$ and~$T$ acts on~$G$ by replacing
every edge with one end-vertex in~$S$ and the other one in~$T$ by a non-edge and vice versa.

We now state some useful facts about how the above operations (and some other ones) influence the clique-width of a graph.
We will use these facts throughout the paper.
Let $k\geq 0$ be a constant and let~$\gamma$ be some graph operation.
We say that a graph class~${\cal G'}$ is {\em $(k,\gamma)$-obtained} from a graph class~${\cal G}$
if the following two conditions hold:
\begin{enumerate}
\item every graph in~${\cal G'}$ is obtained from a graph in~${\cal G}$ by performing~$\gamma$ at most~$k$ times, and
\item for every $G\in {\cal G}$ there exists at least one graph
in~${\cal G'}$ obtained from~$G$ by performing~$\gamma$ at most~$k$ times.
\end{enumerate}
\noindent
We say that~$\gamma$ {\em preserves} boundedness of clique-width if
for any finite constant~$k$ and any graph class~${\cal G}$, any graph class~${\cal G}'$ that is $(k,\gamma)$-obtained from~${\cal G}$
has bounded clique-width if and only if~${\cal G}$ has bounded clique-width.
\begin{enumerate}[\bf F{a}ct 1.]
\item \label{fact:del-vert} Vertex deletion preserves boundedness of clique-width~\cite{LR04}.\\[-1.3em]

\item \label{fact:comp} Subgraph complementation preserves boundedness of clique-width~\cite{KLM09}.\\[-1.3em]

\item \label{fact:bip} Bipartite complementation preserves boundedness of clique-width~\cite{KLM09}.\\[-1.6em]

\end{enumerate}

\begin{lemma}[\cite{CO00}]\label{lem:prime-cw}
Let~$G$ be a graph and let~${\cal P}$ be the set of all induced subgraphs of~$G$ that are prime.
Then $\cw(G)=\max_{H \in {\cal P}}\cw(H)$.
\end{lemma}

\subsection{Well-quasi-orderability}
A {\em quasi order}~$\leq$ on a set~$X$ is a reflexive, transitive binary relation.
Two elements $x,y \in X$ in this quasi-order are {\em comparable} if $x \leq y$ or $y \leq x$, otherwise they are {\em incomparable}.
A set of elements in a quasi-order is a {\em chain} if every pair of elements is comparable and it is an {\em antichain} if every pair of elements is incomparable.
The quasi-order~$\leq$ is a {\em well-quasi-order}
if any infinite sequence of elements $x_1,x_2,x_3,\ldots$ in~$X$
contains a pair $(x_i,x_j)$ with $x_i \leq x_j$ and $i<j$.
Equivalently, a quasi-order is a well-quasi-order if and only if it has no infinite strictly decreasing sequence $x_1 \gneq x_2 \gneq x_3 \gneq \cdots$ and no infinite antichain.

For an arbitrary set~$M$, let~$M^*$ denote the set of finite sequences of elements of~$M$.
A quasi-order~$\leq$ on~$M$ defines a quasi-order~$\leq^*$ on~$M^*$ as follows:
$(a_1,\ldots,a_m) \leq^* (b_1,\ldots,b_n)$ if and only if there is a sequence of integers $i_1,\ldots,i_m$ with $1 \leq i_1<\cdots<i_m \leq n$ such that $a_j \leq b_{i_j}$ for $j \in \{1,\ldots,m\}$.
We call~$\leq^*$ the {\em subsequence relation}.

\begin{lemma}[Higman's Lemma~\cite{Higman52}]\label{lem:higman}
If $(M,\leq)$ is a well-quasi-order then ${(M^*,\leq^*)}$ is a well-quasi-order.
\end{lemma}
To define the notion of labelled induced subgraphs, let us consider an arbitrary quasi-order $(W,\leq)$.
We say that~$G$ is a {\em labelled} graph if each vertex~$v$ of~$G$ is equipped with an element $l_G(v)\in W$ (the {\em label} of~$v$).
Given two labelled graphs~$G$ and~$H$, we say that~$G$ is a {\em labelled induced subgraph} of~$H$ if~$G$ is isomorphic to an induced subgraph of~$H$ and there is an isomorphism that maps each vertex~$v$ of~$G$ to a vertex~$w$ of~$H$ with $l_G(v)\leq l_H(w)$. 
Clearly, if $(W,\leq)$ is a well-quasi-order, then a graph class~$X$ cannot contain an infinite sequence of labelled graphs that is strictly-decreasing with respect to the labelled induced subgraph relation.
We therefore say that a graph class~$X$ is well-quasi-ordered by the {\em labelled} induced subgraph relation if it contains no infinite antichains of labelled graphs whenever $(W,\leq)$ is a {\em well}-quasi-order.
Such a class is readily seen to also be well-quasi-ordered by the induced subgraph relation.

Daligault, Rao and Thomass{\'e}~\cite{DRT10} showed that every hereditary class of graphs that is well-quasi-ordered by the labelled induced subgraph relation is defined by a finite set of forbidden induced subgraphs.
Korpelainen, Lozin and Razgon~\cite{KLR13} conjectured that if a hereditary class of graphs~${\cal G}$ is defined by a finite set of forbidden induced subgraphs, then~${\cal G}$ is well-quasi-ordered by the induced subgraph relation if and only if it is well-quasi-ordered by the labelled induced subgraph relation.
Brignall, Engen and Vatter~\cite{BEV18} recently found a class~${\cal G}^*$ with 14 forbidden induced subgraphs that is a counterexample for this conjecture, that is ${\cal G}^*$ is well-quasi-ordered by the induced subgraph relation but not by the labelled induced subgraph relation.  However, so far, all known results for bigenic graph classes, including those in this paper, verify the conjecture for bigenic graph classes.

Similarly to the notion of preserving boundedness of clique-width, we say that a graph operation~$\gamma$ {\em preserves} well-quasi-orderability by
the labelled induced subgraph relation if for any finite constant~$k$ and any graph class~${\cal G}$, any graph class~${\cal G}'$ that is $(k,\gamma)$-obtained from~${\cal G}$ is well-quasi-ordered by this relation if and only if~${\cal G}$ is.

\begin{lemma}\cite{DLP16}\label{lem:lwqo-operations}
The following operations preserve well-quasi-orderability by the labelled induced subgraph relation:
\begin{enumerate}[(i)]
\renewcommand{\theenumi}{\thelemma.(\roman{enumi})}
\renewcommand{\labelenumi}{(\roman{enumi})}
\item\label{lem:subgraph-complementation} Subgraph complementation,
\item\label{lem:bipartite-complementation} Bipartite complementation and
\item\label{lem:adding-vertices} Vertex deletion.
\end{enumerate}
\end{lemma}

\begin{sloppypar}
\begin{lemma}[\cite{AL15}]\label{lem:prime-wqo}
A hereditary class~$X$ of graphs is well-quasi-ordered by the labelled induced subgraph relation if and only if the set of prime graphs in~$X$ is.
In particular, $X$ is well-quasi-ordered by the labelled induced subgraph relation if and only if the set of connected graphs in~$X$ is.  
\end{lemma}
\begin{lemma}[\cite{AL15,KV11}]\label{lem:P7S123-free-bip-wqo}
$(P_7,S_{1,2,3})$-free bipartite graphs are well-quasi-ordered by the labelled induced subgraph relation.
\end{lemma}
Let~$(L_1,\leq_1)$ and~$(L_2,\leq_2)$ be well-quasi-orders.
We define the {\em Cartesian Product} $(L_1,\leq_1\nobreak) \times (L_2,\leq_2)$ of these well-quasi-orders as the order $(L,\leq_L)$ on the set $L:=L_1\times L_2$ where $(l_1,l_2) \leq_L (l'_1,l'_2)$ if and only if $l_1 \leq_1 l'_1$ and $l_2 \leq_2 l'_2$.
Lemma~\ref{lem:higman} implies that $(L,\leq_L)$ is also a well-quasi-order.
If~$G$ has a labelling with elements of~$L_1$ and of~$L_2$, say $l_1:V(G) \rightarrow L_1$ and $l_2:V(G) \rightarrow L_2$, we can construct the {\em combined labelling} in $(L_1,\leq_1) \times (L_2,\leq_2)$ that labels each vertex~$v$ of~$G$ with the label $(l_1(v),l_2(v))$.
\end{sloppypar}

\begin{lemma}\label{lem:special-labels}
Fix a well-quasi-order~$(L_1,\leq_1)$ that has at least one element.
Let~$X$ be a class of graphs.
For each $G \in X$ fix a labelling $l^1_G: V(G) \rightarrow L_1$.
Then~$X$ is well-quasi-ordered by the labelled induced subgraph relation if and only if for every well-quasi-order~$(L_2,\leq_2)$ and every labelling of the graphs in~$X$ by this order, the combined labelling in $(L_1,\leq_1) \times (L_2,\leq_2)$ obtained from these labellings also results in a well-quasi-ordered set of labelled graphs.
\end{lemma}

\begin{proof}
If~$X$ is well-quasi-ordered by the labelled induced subgraph relation then by definition it is well-quasi-ordered when labelled with labels from these combined labellings obtained from these well-quasi-orders.
If~$X$ is not well-quasi-ordered by the labelled induced subgraph relation then there must be a well-quasi-order $(L_2,\leq_2)$ and an infinite set of graphs $G_1,G_2,\ldots$ whose vertices are labelled with elements of~$L_2$ such that these graphs form an infinite labelled antichain.
For each graph~$G_i$, replace the label~$l$ on vertex~$v$ by $(l^1_{G_1}(v),l)$.
The graphs are now labelled with elements of the well-quasi-order $(L_1,\leq_1) \times (L_2,\leq_2)$ and result in an infinite labelled antichain of graphs labelled with such combined labellings.
This completes the proof.\qed
\end{proof}

\subsection{$k$-uniform Graphs}
For an integer $k\geq 1$, a graph~$G$ is $k$-{\em uniform} if there is a symmetric square $0,1$ matrix~$K$ of order~$k$ and a graph~$F_k$ on vertices
$1, 2, \ldots , k$ such that $G\in {\cal P}(K, F_k)$, where ${\cal P}(K, F_k)$ is the graph class defined as follows.
Let~$H$ be the disjoint union of infinitely many copies of~$F_k$.
For $i = 1,\ldots,k$, let~$V_i$ be the subset of~$V(H)$ containing vertex~$i$ from each copy of~$F_k$. Construct from~$H$ an infinite graph~$H(K)$ on the same vertex set by applying a subgraph complementation to~$V_i$ if and only if $K(i,i)=1$ and by applying a  bipartite complementation to a pair $V_i,V_j$ if and only if $K(i,j)=1$.
Thus, two vertices $u\in V_i$ and $v\in V_j$ are adjacent in~$H(K)$ if and only if $uv \in E(H)$ and $K(i, j) = 0$ or $uv \notin E(H)$ and $K(i, j) = 1$. Then, ${\cal P}(K, F_k)$ is the hereditary class consisting of all the finite induced subgraphs of~$H(K)$.
The minimum~$k$ such that~$G$ is $k$-uniform is the {\em uniformicity} of~$G$. 
The second of the next two lemmas follows directly from the above definitions.

The following result was proved by Korpelainen and Lozin. 
The class of disjoint unions of cliques is a counterexample for the reverse implication.

\begin{lemma}[\cite{KL11}]\label{lem:uniform-wqo}
Any class of graphs of bounded uniformicity is well-quasi-ordered by the labelled induced subgraph relation.
\end{lemma}

The following lemma follows directly from the definition of clique-width and the definition of $k$-uniform graphs (see also~\cite{ALR09} for a more general result).

\begin{lemma}\label{lem:uniform-cw}
Every $k$-uniform graph has clique-width at most~$2k$.
\end{lemma}

\section{Partitioning $3$-Partite Graphs}\label{sec:partitioning}

In Section~\ref{s-deco}, we first introduce a graph decomposition on $3$-partite graphs.
We then show how to extend results on bounded clique-width or well-quasi-orderability by the labelled induced subgraph relation from bipartite graphs in an arbitrary hereditary class of graphs to the $3$-partite graphs in this class that are decomposable in this way.
In Section~\ref{sec:curious}, we then give sufficient conditions for a $3$-partite graph to have such a decomposition.

\subsection{The Decomposition}\label{s-deco}

\begin{sloppypar}
Let~$G$ be a $3$-partite graph given with a partition of its vertex set into three independent sets $V_1$, $V_2$ and~$V_3$.
Suppose that each set~$V_i$ can be partitioned into sets~$V_i^0,\ldots,V_i^\ell$ such that, taking subscripts modulo~$3$:
\begin{itemize}
\item for $i \in \{1,2,3\}$ if $j<k$ then~$V_i^j$ is complete to~$V_{i+1}^k$ and anti-complete to~$V_{i+2}^k$.
\end{itemize}
For $i \in \{0,\ldots,\ell\}$ let $G^i = G[V_1^i \cup V_2^i \cup V_3^i]$.
We say that the graphs~$G^i$ are the {\em slices} of~$G$. If the slices belong to some graph class~$X$, then we say that~$G$ can be {\em partitioned into slices from}~$X$; see \figurename~\ref{fig:slices-order} for an example.
\end{sloppypar}

\begin{figure}[h!]
\begin{center}
\scalebox{1.0}{
\begin{tikzpicture}[scale=0.95]
\draw[rounded corners, color=black, thick] (-1,0) rectangle (1,8);
\draw[rounded corners, color=black, thick] (3,0) rectangle (5,8);
\draw[rounded corners, color=black, thick] (7,0) rectangle (9,8);

\draw[rounded corners, color=black, thick, dashed] (-0.5,0.5) rectangle (8.5,1.5);
\draw[rounded corners, color=black, thick, dashed] (-0.5,2.5) rectangle (8.5,3.5);
\draw[rounded corners, color=black, thick, dashed] (-0.5,4.5) rectangle (8.5,5.5);
\draw[rounded corners, color=black, thick, dashed] (-0.5,6.5) rectangle (8.5,7.5);

\draw (0,0) node [label={[label distance=-0pt]-90:$V_1$}] {};
\draw (4,0) node [label={[label distance=-0pt]-90:$V_2$}] {};
\draw (8,0) node [label={[label distance=-0pt]-90:$V_3$}] {};
\draw (-0.5,7) node [label={[label distance=-5.5pt]180:$G^0$}] {};
\draw (-0.5,5) node [label={[label distance=-5.5pt]180:$G^1$}] {};
\draw (-0.5,3) node [label={[label distance=-5.5pt]180:$G^2$}] {};
\draw (-0.5,1) node [label={[label distance=-5.5pt]180:$G^3$}] {};

\coordinate (x_1) at (0,1) ;
\coordinate (y_1) at (4,1.2) ;
\coordinate (z_1) at (8,1) ;
\coordinate (x_2) at (0,3) ;
\coordinate (y_2) at (4,3.2) ;
\coordinate (z_2) at (8,3) ;
\coordinate (x_3) at (0,5) ;
\coordinate (y_3) at (4,5.2) ;
\coordinate (z_3) at (8,5) ;
\coordinate (x_4) at (0,7) ;
\coordinate (y_4) at (4,7.2) ;
\coordinate (z_4) at (8,7) ;

\draw [fill=black] (x_1) circle (1.5pt) ;
\draw [fill=black] (y_1) circle (1.5pt) ;
\draw [fill=black] (z_1) circle (1.5pt) ;
\draw [fill=black] (x_2) circle (1.5pt) ;
\draw [fill=black] (y_2) circle (1.5pt) ;
\draw [fill=black] (z_2) circle (1.5pt) ;
\draw [fill=black] (x_3) circle (1.5pt) ;
\draw [fill=black] (y_3) circle (1.5pt) ;
\draw [fill=black] (z_3) circle (1.5pt) ;
\draw [fill=black] (x_4) circle (1.5pt) ;
\draw [fill=black] (y_4) circle (1.5pt) ;
\draw [fill=black] (z_4) circle (1.5pt) ;

\draw (x_1) -- (y_1) -- (z_1);
\draw (x_2) -- (y_2) -- (z_2);
\draw (x_3) -- (y_3) -- (z_3);
\draw (x_4) -- (y_4) -- (z_4);

\draw (x_4) -- (y_3);
\draw (x_4) -- (y_2);
\draw (x_4) -- (y_1);
\draw (x_3) -- (y_2);
\draw (x_3) -- (y_1);
\draw (x_2) -- (y_1);

\draw (y_4) -- (z_3);
\draw (y_4) -- (z_2);
\draw (y_4) -- (z_1);
\draw (y_3) -- (z_2);
\draw (y_3) -- (z_1);
\draw (y_2) -- (z_1);

\draw (z_4) -- (x_3);
\draw (z_4) -- (x_2);
\draw (z_4) -- (x_1);
\draw (z_3) -- (x_2);
\draw (z_3) -- (x_1);
\draw (z_2) -- (x_1);

\end{tikzpicture}
}
\end{center}
\caption{\label{fig:slices-order}A $3$-partite graph partitioned into slices $G^0,\ldots,G^3$ isomorphic to~$P_3$.}
\end{figure}
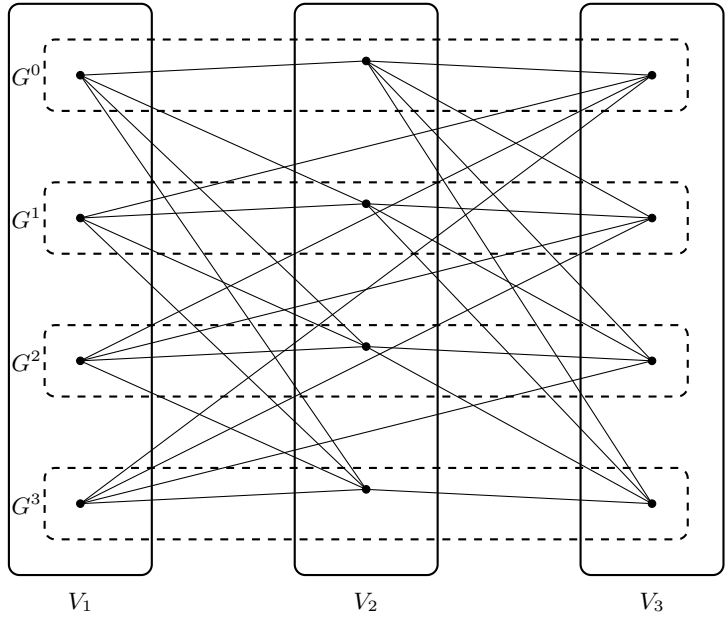

\begin{lemma}\label{lem:slice-cw}
If~$G$ is a $3$-partite graph that can be partitioned into slices of clique-width at most~$k$ then~$G$ has clique-width at most~$\max(3k,6)$.
\end{lemma}
\begin{proof}
Since every slice~$G^j$ of~$G$ has clique-width at most~$k$, it can be constructed using the labels $1,\ldots,k$.
Applying relabelling operations if necessary, we may assume that at the end of this construction, every vertex receives the label~$1$.
We can modify this construction so that we use the labels $1_1,\ldots,k_1,1_2,\ldots,k_2,1_3,\ldots,k_3$ instead, in such a way that at all points in the construction, for each $i \in \{1,2,3\}$ every constructed vertex in~$V_i$ has a label in $\{1_i,\ldots,k_i\}$.
To do this we replace:
\begin{sloppypar}
\begin{itemize}
\item creation operations~$i(v)$ by~$i_j(v)$ if $v \in V_j$,
\item relabel operations~$\rho_{j \rightarrow k}()$ by $\rho_{j_1 \rightarrow k_1}(\rho_{j_2 \rightarrow k_2}(\rho_{j_3 \rightarrow k_3}()))$ and
\item join operations~$\eta_{j,k}()$ by\\\hspace*{2em} $\eta_{j_1,k_1}(\eta_{j_1,k_2}(\eta_{j_1,k_3}(\eta_{j_2,k_1}(\eta_{j_2,k_2}(\eta_{j_2,k_3}(\eta_{j_3,k_1}(\eta_{j_3,k_2}(\eta_{j_3,k_3}()))))))))$.
\end{itemize}
\end{sloppypar}
\noindent
This modified construction uses~$3k$ labels and at the end of it, every vertex in~$V_i$ is labelled with label~$1_i$.
We may do this for every slice~$G^j$ of~$G$ independently.
We now show how to use these constructed slices to construct $G[V(G^0) \cup \cdots \cup V(G^j)]$ using six labels in such a way that every vertex in~$V_i$ is labelled with label~$1_i$.
We do this by induction.
If $j=0$ then $G[V(G^0)]=G^0$, so we are done.
If $j>0$ then by the induction hypothesis, we can construct $G[V(G^0) \cup \cdots \cup V(G^{j-1})]$ in this way.
Consider the copy of~$G^j$ constructed earlier and relabel its vertices using the operations $\rho_{1_1 \rightarrow 2_1}$, $\rho_{1_2 \rightarrow 2_2}$ and $\rho_{1_3 \rightarrow 2_3}$ so that in this copy of~$G^j$, every vertex in~$V_i$ is labelled~$2_i$.
Next take the disjoint union of the obtained graph with the constructed $G[V(G^0) \cup \cdots \cup V(G^{j-1})]$.
Then, apply join operations $\eta_{1_1,2_2}$, $\eta_{1_2,2_3}$ and~$\eta_{1_3,2_1}$.
Finally, apply the relabelling operations $\rho_{2_1 \rightarrow 1_1}$, $\rho_{2_2 \rightarrow 1_2}$ and $\rho_{2_3 \rightarrow 1_3}$.
This constructs $G[V(G^0) \cup \cdots \cup V(G^j)]$ in such a way that every vertex  in~$V_i$ is labelled with~$1_i$.
By induction, it follows that~$G$ has clique-width at most~$\max(3k,6)$.\qed
\end{proof}

\begin{lemma}\label{lem:slice-wqo}
Let~$X$ be a hereditary graph class containing a class~$Z$.
Let~$Y$ be the set of $3$-partite graphs in~$X$ that can be partitioned into slices from~$Z$.
If~$Z$ is well-quasi-ordered by the labelled induced subgraph relation then so is~$Y$.
\end{lemma}
\begin{proof}
For each graph~$G$ in~$Y$, we may fix a partition into independent sets $(V_1,V_2,V_3)$ with respect to which the graph can be partitioned into slices from~$Z$.
Let $(L_1,\leq_1)$ be the well-quasi-order with $L_1=\{1,2,3\}$ in which every pair of distinct elements is incomparable.
By Lemma~\ref{lem:special-labels}, we need only consider labellings of graphs in~$G$ of the form $(i,l(v))$ where $v \in V_i$ and~$l(v)$ belongs to an arbitrary well-quasi-order~$L$.
Suppose~$G$ can be partitioned into slices $G^1,\ldots,G^k$, with vertices labelled as in~$G$.
The slices along with the labellings completely describe the edges in~$G$.
Suppose~$H$ is another such graph, partitioned into slices $H^1,\ldots,H^k$.
If $(H^1,\ldots,H^\ell)$ is smaller than $(G^1,\ldots,G^k)$ under the subsequence relation, then~$H$ is an induced subgraph of~$G$.
The result follows by Lemma~\ref{lem:higman}.\qed
\end{proof}

\subsection{Curious Graphs}\label{sec:curious}
Let~$G$ be a $3$-partite graph given together with a partition of its vertex set into three independent sets $V_1$, $V_2$ and~$V_3$.
An induced~$K_3$ or~$3P_1$ in~$G$ is {\em rainbow} if it has exactly one vertex in each set~$V_i$.
We  say that~$G$ is {\em curious with respect to the partition $(V_1,V_2,V_3)$} if it contains no rainbow~$K_3$ or~$3P_1$ when its vertex set is partitioned in this way.
We say that~$G$ is {\em curious} if there is a partition $(V_1,V_2,V_3)$ with respect to which~$G$ is curious.
In this section we will prove that given a hereditary class~$X$, if the bipartite graphs in~$X$ are well-quasi-ordered by the labelled induced subgraph relation or have bounded clique-width, then the same is true for the curious graphs in~$X$.

A linear order $(x_1,x_2,\ldots,x_k)$ of the vertices of an independent set~$I$ is
\begin{itemize}
\item {\em increasing} if $i<j$ implies $N(x_i)\subseteq N(x_j)$,
\item {\em decreasing} if $i<j$ implies $N(x_i)\supseteq N(x_j)$,
\item {\em monotone} if it is either increasing or decreasing.
\end{itemize}
Bipartite graphs that are $2P_2$-free are also known as bipartite {\em chain} graphs.
It is and it is well known (and easy to verify) that a bipartite graph~$G$ is $2P_2$-free if and only if the vertices in each independent set of the bipartition admit a monotone ordering.
Suppose~$G$ is a curious graph with respect to some partition $(V_1,V_2,V_3)$.
We say that (with respect to this partition) the graph~$G$ is a curious graph of
{\em type~$t$} if exactly~$t$ of the graphs $G[V_1\cup V_2]$, $G[V_1\cup V_3]$ and~$G[V_2\cup V_3]$ contain an induced~$2P_2$. 

\subsubsection{Curious Graphs of Type~$0$ and~$1$.}
Note that if~$G$ is a curious graph of type~$0$ or~$1$ with respect to the partition $(V_1,V_2,V_3)$ then without loss of generality, we may assume that $G[V_1\cup V_2]$ and $G[V_1\cup V_3]$ are both $2P_2$-free.

\begin{lemma}\label{lem:type01-lin-order}
Let~$G$ be a curious graph with respect to $(V_1,V_2,V_3)$, such that $G[V_1\cup\nobreak V_2]$ and $G[V_1\cup V_3]$ are both $2P_2$-free.
Then the vertices of~$V_1$ admit a linear ordering which is decreasing in $G[V_1\cup V_2]$ and increasing in $G[V_1\cup V_3]$.
\end{lemma} 

\begin{proof}
For a set $S \subseteq V$, we use $N_S(u):=N(u) \cap S$ to denote the set of vertices in~$S$ that are adjacent to~$u$.
We may choose a linear order $x_1,\ldots,x_\ell$ of the vertices of~$V_1$ according to their neighbourhood in~$V_2$, breaking ties according to their neighbourhood in~$V_3$ i.e. an order such that:
\begin{enumerate}[(i)]
\renewcommand{\theenumi}{(\roman{enumi})}
\renewcommand{\labelenumi}{(\roman{enumi})}
\item \label{prop5:V2-nhd-differs}if $N_{V_2}(x_i) \supsetneq N_{V_2}(x_j)$ then $i <j$ and
\item \label{prop5:V2nhd-equal-V3-nhd-differs}if $N_{V_2}(x_i) = N_{V_2}(x_j)$ and $N_{V_3}(x_i) \subsetneq N_{V_3}(x_j)$ then $i<j$.
\end{enumerate}
Clearly such an ordering is decreasing in $G[V_1\cup V_2]$.

Suppose, for contradiction, that this order is not increasing in $G[V_1\cup V_3]$.
Then there must be indices $i<j$ such that $N_{V_3}(x_i) \supsetneq N_{V_3}(x_j)$.
Then $N_{V_2}(x_i) \neq N_{V_2}(x_j)$ by Property~\ref{prop5:V2nhd-equal-V3-nhd-differs}.
By Property~\ref{prop5:V2-nhd-differs} it follows that $N_{V_2}(x_i) \supsetneq N_{V_2}(x_j)$.
This means that there are vertices $y \in N_{V_2}(x_i) \setminus N_{V_2}(x_j)$ and $z \in N_{V_3}(x_i) \setminus N_{V_3}(x_j)$.
Now if~$y$ is adjacent to~$z$ then $G[x_i,y,z]$ is a rainbow~$K_3$ and if~$y$ is non-adjacent to~$z$ then $G[x_j,y,z]$ is a rainbow~$3P_1$.
This contradiction implies that the order is indeed decreasing in $G[V_1\cup V_3]$, which completes the proof.\qed
\end{proof}

\begin{sloppypar}
\begin{lemma}\label{lem:cur-type01-slice-bip}
If~$G$ is a curious graph of type~$0$ or~$1$ with respect to a partition $(V_1,V_2,V_3)$ then~$G$ can be partitioned into slices that are bipartite.
\end{lemma}
\end{sloppypar}
\begin{proof}
Let $x_1,\ldots,x_\ell$ be a linear order on~$V_1$ satisfying Lemma~\ref{lem:type01-lin-order}.
Let $V_1^0=\emptyset$ and for $i \in \{1,\ldots,\ell\}$, let $V_1^i=\{x_i\}$.
We partition~$V_2$ and~$V_3$ as follows.
For $i \in \{0,\ldots,\ell\}$, let $V_2^i = \{y \in V_2 \; | \; x_jy \in E(G) \; \textrm{if and only if} \; j\leq i\}$.
For $i \in \{0,\ldots,\ell\}$, let $V_3^i = \{z \in V_3 \; | \; x_jz \notin E(G) \; \textrm{if and only if} \; j\leq i\}$.
In particular, note that the vertices of $V_2^\ell \cup V_3^0$ and~$V_2^0 \cup V_3^\ell$ are complete and anti-complete to~$V_1$, respectively.
The following properties hold:
\begin{itemize}
\item If $j<k$ then~$V_1^j$ is complete to~$V_2^k$ and anti-complete to~$V_3^k$.
\item If $j>k$ then~$V_1^j$ is anti-complete to~$V_2^k$ and complete to~$V_3^k$.
\end{itemize}
If $j<k$ and $y \in V_2^j$ is non-adjacent to $z \in V_3^k$ then $G[x_k,y,z]$ is a rainbow~$3P_1$, a contradiction.
If $j>k$ and $y \in V_2^j$ is adjacent to $z \in V_3^k$ then $G[x_j,y,z]$ is a rainbow~$K_3$, a contradiction.
It follows that:
\begin{itemize}
\item If $j<k$ then~$V_2^j$ is complete to~$V_3^k$.
\item If $j>k$ then~$V_2^j$ is anti-complete to~$V_3^k$.
\end{itemize}

For $i \in \{0,\ldots,\ell\}$, let $G^i=G[V_1^i \cup V_2^i \cup V_3^i]$.
The above properties about the edges between the sets~$V_j^i$ show that~$G$ can be partitioned into the slices $G^0,\ldots,G^\ell$.
Now, for each $i \in \{0,\ldots,\ell\}$, $V_1^i$ is anti-complete to~$V_3^i$, so every slice~$G^i$ is bipartite.
This completes the proof.\qed
\end{proof}

\subsubsection{Curious Graphs of Type~$2$ and~$3$.}

\begin{lemma}\label{lem:cur-type23-slice-type12}
Fix $t \in \{2,3\}$.
If~$G$ is a curious graph of type~$t$ with respect to a partition $(V_1,V_2,V_3)$ then~$G$ can be partitioned into slices of type at most $t-1$.
\end{lemma}

\begin{proof}
\setcounter{ctrclaim}{0}
Fix $t \in \{2,3\}$ and let~$G$ be a curious graph of type~$t$ with respect to a partition $(V_1,V_2,V_3)$.
We may assume that $G[V_1\cup V_2]$ contains an induced~$2P_2$.

\clm{\label{clm5:nbr-in-2P2a}Given a~$2P_2$ in $G[V_1\cup V_2]$, every vertex of~$V_3$ has exactly two neighbours in the~$2P_2$ and these neighbours either both lie in~$V_1$ or both lie in~$V_2$.}
Let $x_1,x_2 \in V_1$ and $y_1,y_2 \in V_2$ induce a~$2P_2$ in~$G$ such that~$x_1$ is adjacent to~$y_1$ but not to~$y_2$ and~$x_2$ is adjacent to~$y_2$ but not to~$y_1$.
Consider $z \in V_3$.
For $i \in \{1,2\}$, if~$z$ is adjacent to both~$x_i$ and~$y_i$, then $G[x_i,y_i,z]$ is a rainbow~$K_3$, a contradiction.
Therefore~$z$ can have at most one neighbour in $\{x_1,y_1\}$ and at most one neighbour in~$\{x_2,y_2\}$.
If~$z$ is non-adjacent to~$x_1$ and~$y_2$ then $G[x_1,y_2,z]$ is a rainbow~$3P_1$, a contradiction.
Therefore, $z$ can have at most one non-neighbour in $\{x_1,y_2\}$, and similarly~$z$ have have at most one non-neighbour in $\{x_2,y_1\}$.
Therefore, if~$z$ is adjacent to~$x_1$, then it must be non-adjacent to~$y_1$, so it must be adjacent to~$x_2$, so it must be non-adjacent to~$y_2$.
Similarly, if~$z$ is non-adjacent to~$x_1$ then it must be adjacent to~$y_2$, so it must be non-adjacent to~$x_2$, so it must be adjacent to~$y_1$.
Hence Claim~\ref{clm5:nbr-in-2P2a} follows.

\medskip
\noindent
Consider a maximal set $\{H^1,\ldots,H^q\}$ of vertex-disjoint sets that induce copies  of~$2P_2$ in $G[V_1\cup V_2]$.
We say that a vertex of~$V_3$ {\em distinguishes} two graphs~$G[H^i]$ and~$G[H^j]$ if its neighbours in~$H^i$ and~$H^j$ do not belong to the same set~$V_k$.
We group these sets~$H^i$ into {\em blocks} $B^1,\ldots,B^p$ that are not distinguished by any vertex of~$V_3$.
In other words, every~$B^i$ contains at least one~$2P_2$ and every vertex of~$V_3$ is complete to one of the sets $B^i \cap V_1$ and $B^i \cap V_2$ and anti-complete to the other.
For $j \in \{1,2\}$, let $B^i_j=B^i \cap V_j$.
We define a relation~$<_B$ on the blocks as follows:
\begin{itemize}
\item $B^i<_B B^j$ holds if~$B_1^i$ is complete to~$B_2^j$, while~$B_2^i$ is anti-complete to~$B_1^j$.
\end{itemize}
For distinct blocks $B^i$, $B^j$ at most one of $B^i<_B B^j$ and $B^j<_B B^i$ can hold.

\clm{\label{clm5:Bi-Bj-z-diffa}Let~$B^i$ and~$B^j$ be distinct blocks.
There is a vertex $z \in V_3$ that differentiates~$B^i$ and~$B^j$.
If~$z$ is complete to $B^i_2 \cup B^j_1$ and anti-complete to $B^i_1 \cup B^j_2$ then $B^i<_B B^j$ (see also \figurename~\ref{fig:block-order2}).
If~$z$ is complete to $B^j_2 \cup B^i_1$ and anti-complete to $B^j_1 \cup B^i_2$ then $B^j<_B B^i$.}
By definition of the blocks~$B^i$ and~$B^j$ there must be a vertex $z \in V_3$ that distinguishes them.
Without loss of generality we may assume that~$z$ is complete to $B_2^i\cup B_1^j$ and anti-complete to $B_1^i\cup B_2^j$.
It remains to show that $B^i<_B B^j$.
If $y \in B_2^i$ is adjacent to $z \in B_1^j$ then $G[x,y,z]$ is a rainbow~$K_3$, a contradiction.
If $y \in B_1^i$ is non-adjacent to $z \in B_2^j$ then $G[x,y,z]$ is a rainbow~$3P_1$, a contradiction.
Therefore~$B_2^i$ is anti-complete to~$B_1^j$ and~$B_1^i$ is complete to~$B_2^j$.
It follows that $B^i<_B B^j$.
Claim~\ref{clm5:Bi-Bj-z-diffa} follows by symmetry.

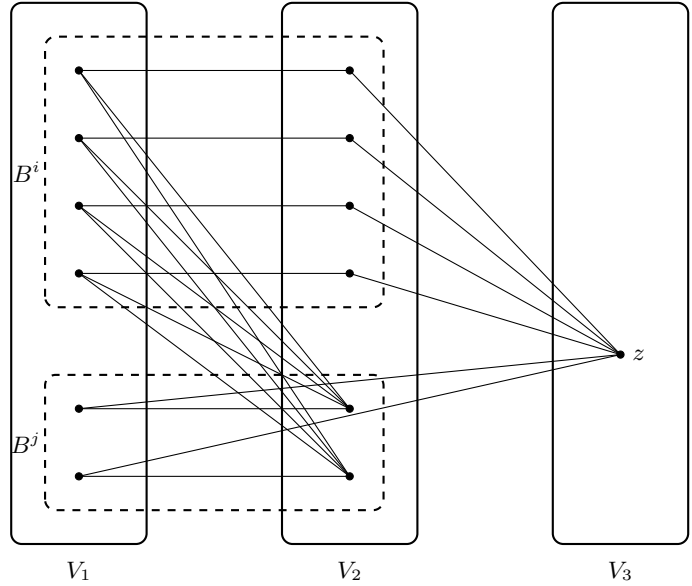
\begin{figure}
\begin{center}
\begin{tikzpicture}[scale=0.9]
\draw[rounded corners, color=black, thick] (-1,0) rectangle (1,8);
\draw[rounded corners, color=black, thick] (3,0) rectangle (5,8);
\draw[rounded corners, color=black, thick] (7,0) rectangle (9,8);

\draw[rounded corners, color=black, thick, dashed] (-0.5,0.5) rectangle (4.5,2.5);
\draw[rounded corners, color=black, thick, dashed] (-0.5,3.5) rectangle (4.5,7.5);

\coordinate (x_1) at (0,1) ;
\coordinate (x_2) at (0,2) ;
\coordinate (x_4) at (0,4) ;
\coordinate (x_5) at (0,5) ;
\coordinate (x_6) at (0,6) ;
\coordinate (x_7) at (0,7) ;
\coordinate (y_1) at (4,1) ;
\coordinate (y_2) at (4,2) ;
\coordinate (y_4) at (4,4) ;
\coordinate (y_5) at (4,5) ;
\coordinate (y_6) at (4,6) ;
\coordinate (y_7) at (4,7) ;
\coordinate (z) at (8,2.8) ;

\draw (0,0) node [label={[label distance=-0pt]-90:$V_1$}] {};
\draw (4,0) node [label={[label distance=-0pt]-90:$V_2$}] {};
\draw (8,0) node [label={[label distance=-0pt]-90:$V_3$}] {};
\draw (-0.5,5.5) node [label={[label distance=-5pt]180:$B^i$}] {};
\draw (-0.5,1.5) node [label={[label distance=-5pt]180:$B^j$}] {};
\draw (z) node [label={[label distance=-2pt]0:$z$}] {};

\draw [fill=black] (x_1) circle (1.5pt) ;
\draw [fill=black] (x_2) circle (1.5pt) ;
\draw [fill=black] (x_4) circle (1.5pt) ;
\draw [fill=black] (x_5) circle (1.5pt) ;
\draw [fill=black] (x_6) circle (1.5pt) ;
\draw [fill=black] (x_7) circle (1.5pt) ;
\draw [fill=black] (y_1) circle (1.5pt) ;
\draw [fill=black] (y_2) circle (1.5pt) ;
\draw [fill=black] (y_4) circle (1.5pt) ;
\draw [fill=black] (y_5) circle (1.5pt) ;
\draw [fill=black] (y_6) circle (1.5pt) ;
\draw [fill=black] (y_7) circle (1.5pt) ;
\draw [fill=black] (z)   circle (1.5pt) ;

\draw (x_1) -- (y_1);
\draw (x_2) -- (y_2);
\draw (x_4) -- (y_4);
\draw (x_5) -- (y_5);
\draw (x_6) -- (y_6);
\draw (x_7) -- (y_7);

\draw (x_1) -- (z);
\draw (x_2) -- (z);
\draw (z) -- (y_4);
\draw (z) -- (y_5);
\draw (z) -- (y_6);
\draw (z) -- (y_7);

\draw (x_4) -- (y_1);
\draw (x_5) -- (y_1);
\draw (x_6) -- (y_1);
\draw (x_7) -- (y_1);
\draw (x_4) -- (y_2);
\draw (x_5) -- (y_2);
\draw (x_6) -- (y_2);
\draw (x_7) -- (y_2);

\end{tikzpicture}
\end{center}
\caption{\label{fig:block-order2}Two blocks~$B^i$ and~$B^j$ with $B^i <_B B^j$ and a vertex $z \in V_3$ differentiating them.}
\end{figure}

\clm{\label{c-3a}The relation~$<_B$ is transitive.}
Suppose that $B^i <_B B^j$ and $B^j <_B B^k$.
Since~$B^j$ and~$B^k$ are distinct, there must be a vertex $z \in V_3$ that distinguishes them and combining Claim~\ref{clm5:Bi-Bj-z-diffa} with the fact that $B^j <_B B^k$ it follows that~$z$ must be complete to $B_2^j\cup B_1^k$.
Suppose $x \in B^i_1$ and $y \in B^j_2$.
Then~$x$ is adjacent to~$y$, since $B^i <_B B^j$ and~$y$ is adjacent to~$z$ by choice of~$z$.
Now~$x$ is must be non-adjacent to~$z$ otherwise $G[x,y,z]$ would be a rainbow~$K_3$.
Therefore~$z$ is anti-complete to~$B_1^i$, so~$z$ distinguishes~$B^i$ and~$B^k$.
By Claim~\ref{clm5:Bi-Bj-z-diffa} it follows that $B^i <_B B^k$.
This completes the proof of Claim~\ref{c-3a}.

\medskip
\noindent
Combining Claims~\ref{clm5:nbr-in-2P2a}--\ref{c-3a}, we find that~$<_B$ is a linear order on the blocks.
We obtain the following conclusion, which we call the {\em chain property}.

\clmnonewline{\label{clm5:chain-propa}The set of blocks admits a linear order $B^1 <_B B^2 <_B \cdots <_B B^p$ such that 
\begin{enumerate}[(i)]
\renewcommand{\theenumi}{(\roman{enumi})}
\renewcommand{\labelenumi}{(\roman{enumi})}
\item \label{clm5:chain-prop-ia}if~$i<j$ then~$B^i_1$ is complete to~$B^j_2$, while~$B^i_2$ is anti-complete to~$B^j_1$ and
\item \label{clm5:chain-prop-iia}for each $z \in V_3$ there exists an $i\in \{0,\ldots,p\}$ such that
if $j \leq i$ then~$z$ is complete to~$B^j_2$ and anti-complete to~$B^j_1$ and if $j>i$ then~$z$ is anti-complete to~$B^j_2$ and complete to~$B^j_1$.
\end{enumerate}}

\noindent
Next consider the set of vertices in $V_1 \cup V_2$ that do not belong to any set~$B^i$.
Let~$R$ denote this set and note that~$G[R]$ is $2P_2$-free by maximality of the set $\{H^1,\ldots,H^q\}$.
For $i \in \{1,2\}$ let $R_i=R \cap V_i$.

\clm{\label{claim:+a}If $x\in R_1$ has a neighbour in~$B^i_2$, then~$x$ is complete to~$B^{i+1}_2$, and if~$x$ has a non-neighbour in~$B^i_2$, then~$x$ is anti-complete to~$B^{i-1}_2$. 
If $x\in R_2$ has a non-neighbour in $B^i_1$, then $x$ is anti-complete to~$B^{i+1}_1$, and if~$x$ has a neighbour in~$B^i_1$, then~$x$ is complete to~$B^{i-1}_1$.}
Suppose, for contradiction, that $x\in R_1$ is adjacent to $y\in B^i_2$ and non-adjacent to $y'\in B^{i+1}_2$. 
Consider a vertex $z\in V_3$ that distinguishes~$B^i$ and~$B^{i+1}$.
Since $B^i <_B B^{i+1}$, Claim~\ref{clm5:Bi-Bj-z-diffa} implies that~$z$ is complete to $B^{i+1}_1 \cup B^i_2$ and anti-complete to $B^{i+1}_2 \cup B^i_1$.
Now if~$x$ is adjacent to~$z$ then $G[x,y,z]$ is a rainbow~$K_3$ and if~$x$ is non-adjacent to~$z$ then $G[x,y',z]$ is a rainbow~$3P_1$.
It follows that if $x\in R_1$ has a neighbour in~$B^i_2$ then~$x$ is complete to~$B^{i+1}_2$ and if~$x$ has a non-neighbour in~$B^{i+1}_2$, then~$x$ is anti-complete to~$B^i_2$.
The claim follows by symmetry.

\medskip
\noindent
Claim~\ref{claim:+a} allows us to update the sequence of blocks as follows:

\medskip
\noindent
{\bf Update Procedure. }{\em For $i \in \{1,2\}$, if~$R_i$ contains a vertex~$x$ that has both a neighbour~$y$ and a non-neighbour~$y'$ in~$B^j_{3-i}$ for some~$j$, we add~$x$ to the sets~$B^j_i$ and~$B^j$ and remove it from~$R_i$.}

\clm{\label{clm5:update-proc-preserves-chain-propa}Applying the Update Procedure preserves the chain property (see Claim~\ref{clm5:chain-propa}) of the blocks~$B^i$.}
Assume that the chain property holds (possibly after some applications of the Update Procedure).
Without loss of generality, assume $x \in R_1$ has both a neighbour~$y$ and a non-neighbour~$y'$ in some set~$B^j_2$ (the case where $x \in R_2$ follows similarly).
We will show that the chain property continues to hold after adding~$x$ to~$B^j_1$ and~$B_j$ and removing it from~$R_1$.
Recall that every vertex of~$B^j_1$ has the same neighbourhood in~$V_3$ by definition of~$B^j$.
We first show that~$x$ has the same neighbourhood in~$V_3$ as the vertices of~$B^j_1$.
If $z \in V_3$ is non-adjacent to~$x$, but complete to~$B^j_1$, then~$z$ is anti-complete to~$B^j_2$, so $G[x,y',z]$ is a rainbow~$3P_1$, a contradiction.
Similarly, if $z \in V_3$ is adjacent to~$x$, but anti-complete to~$B^j_1$, then~$z$ is complete to~$B^j_2$, so $G[x,y,z]$ is a rainbow~$K_3$, a contradiction.
Therefore~$x$ must have the same neighbourhood in~$V_3$ as the vertices of~$B^j_1$.
By Claim~\ref{clm5:chain-propa}, this means that Property~\ref{clm5:chain-prop-iia} of the chain property is preserved if we apply the Update Procedure with the vertex~$x$.

Now suppose that $y'' \in B^k_2$ for some $k \neq j$.
Then there must be a vertex $z \in V_3$ that differentiates~$B^j$ and~$B^k$.
If $k<j$ then by Claim~\ref{clm5:Bi-Bj-z-diffa} the vertex~$z$ is complete to $B^k_2 \cup B^j_1$ and anti-complete to $B^k_1 \cup B^j_2$, so~$x$ must be non-adjacent to~$y''$, otherwise $G[x,y'',z]$ would be a rainbow~$K_3$.
If $k>j$ then by Claim~\ref{clm5:Bi-Bj-z-diffa} the vertex~$z$ is anti-complete to $B^k_2 \cup B^j_1$ and complete to $B^k_1 \cup B^j_2$, so~$x$ must be adjacent to~$y''$, otherwise $G[x,y'',z]$ would be a rainbow~$3P_1$.
We conclude that Property~\ref{clm5:chain-prop-ia} of the chain property is also preserved if we apply the Update Procedure with the vertex~$x$.
By symmetry and induction this completes the proof of Claim~\ref{clm5:update-proc-preserves-chain-propa}.

\medskip
\noindent
By Claim~\ref{clm5:update-proc-preserves-chain-propa} we may therefore apply the Update Procedure exhaustively, after which the chain property will continue to hold.
Once this procedure is complete, every remaining vertex of~$R_1$ will be either complete or anti-complete to each set~$B^j_2$.
In fact, by Claim~\ref{claim:+a}, we know that for every vertex $x \in R_1$, there is an $i \in \{0,\ldots,p\}$ such that~$x$ has a neighbour in all~$B^j_2$ with $j>i$ (if such a~$j$ exists) and~$x$ has a non-neighbour in all $B^j_2$ with $j\leq i$ (if any such~$j$ exists).
Since~$x$ is complete or anti-complete to each set~$B^j_2$, we obtain the following conclusion:
\begin{itemize}
\item for every vertex $x \in R_1$, there is an $i \in \{0,\ldots,p\}$ such that~$x$ is complete to all~$B^j_2$ with $j>i$ (if such a~$j$ exists) and~$x$ is anti-complete to all $B^j_2$ with $j\leq i$ (if any such~$j$ exists).
We denote the corresponding subset of~$R_1$ by~$Y^i_1$.
\end{itemize}
By symmetry, we also obtain the following:
\begin{itemize}
\item for every vertex $x \in R_2$, there is an $i \in \{0,\ldots,p\}$ such that~$x$ is complete to all~$B^j_1$ with $j \leq i$ (if such a~$j$ exists) and~$x$ is anti-complete to all $B^j_1$ with $j > i$ (if any such~$j$ exists).
We denote the corresponding subset of~$R_2$ by~$Y^i_2$.
\end{itemize}

\noindent
We also partition the vertices of~$V_3$ into~$p+\nobreak 1$ subsets $V_3^0,\ldots,V_3^p$ such that the vertices of~$V_3^j$ are complete to~$B^i_2$ and anti-complete to~$B^i_1$ for $i \leq j$ and complete to~$B^i_1$ and anti-complete to~$B^i_2$ for $i > j$.
(So~$V_3^0$ is complete to~$B^1_i$ for all~$i$ and~$V_3^p$ is complete to~$B^2_i$ for all~$i$).

\clm{\label{claim:-a}For each~$i$, if $j<i$ then~$V_3^i$ is anti-complete to~$Y^j_1$ and complete to~$Y^j_2$, and if $j>i$ then~$V_3^i$ is complete to~$Y^j_1$ and anti-complete to~$Y^j_2$.}
Suppose that $z \in V_3^i$ and $x \in Y^j_1$ and $y \in Y^j_2$ (note that such vertices~$x$ and~$y$ do not exist if~$Y^j_1$ or~$Y^j_2$, respectively, is empty).
First suppose that $j<i$ and choose arbitrary vertices $x' \in B^i_1$, $y' \in B^i_2$.
Note that~$x$ and~$z$ are both complete to~$B_2^i$ and~$y$ and~$z$ are both anti-complete to~$B_1^i$.
Then~$z$ cannot be adjacent to~$x$ otherwise $G[x,y',z]$ would be a rainbow~$K_3$ and~$z$ must be adjacent to~$y$, otherwise $G[x',y,z]$ would be a rainbow~$3P_1$.
Now suppose $i<j$ and choose arbitrary vertices $x' \in B^{i+1}_1$, $y' \in B^{i+1}_2$.
Note that~$x$ and~$z$ are both anti-complete to~$B_2^{i+1}$ and~$y$ and~$z$ are both complete to~$B_1^{i+1}$.
Then~$z$ must be adjacent to~$x$ otherwise $G[x,y',z]$ would be a rainbow~$3P_1$ and~$z$ must be non-adjacent to~$y$, otherwise $G[x',y,z]$ would be a rainbow~$K_3$.
This completes the proof of Claim~\ref{claim:-a}.

\medskip
\begin{sloppypar}
\noindent
Let~$G^i$ denote the subgraph of~$G$ induced by $Y_1^i\cap Y_2^i\cap V_3^i$.
By Claims~\ref{clm5:chain-propa}, \ref{clm5:update-proc-preserves-chain-propa} and~\ref{claim:-a} the graph~$G$ can be partitioned into slices: $G^0,G[B^1],G^1,G[B^2],\ldots,G[B^p],G^p$.
Recall that the graph~$G$ is of type~$t$ and $G[V_1\cup V_2]$ contains an induced~$2P_2$.
Since~$G[Y_1^i \cup Y_2^i]$ is $2P_2$-free (by construction, since the original sequence $H^1,H^2,\ldots,H^q$ of $2P_2$s was maximal), it follows that each~$G^i$ is of type at most~$t-1$.
Furthermore, since each~$G[B_i]$ is bipartite, it forms a curious graph in which the set~$V_3$ is empty, so it has type at most~$1$.
This completes the proof.\qed
\end{sloppypar}
\end{proof}

\subsubsection{Curious Graphs, Clique-width and Well-quasi-orderability.}

We are now ready to state the main result of this section.
\begin{theorem}\label{thm:curious}
Let~$X$ be a hereditary class of graphs.
If the set of bipartite graphs in~$X$ is well-quasi-ordered by the labelled induced subgraph relation or has bounded clique-width, then this property also holds for the set of curious graphs in~$X$.
\end{theorem}

\begin{proof}
Suppose that the class of bipartite graphs in~$X$ is well-quasi-ordered by the labelled induced subgraph relation or has bounded clique-width.
By Lemmas~\ref{lem:slice-cw}, \ref{lem:slice-wqo} and~\ref{lem:cur-type01-slice-bip}, the curious graphs of type at most~$1$ also have this property.
Applying Lemmas~\ref{lem:slice-cw}, \ref{lem:slice-wqo} and~\ref{lem:cur-type23-slice-type12} once, we obtain the same conclusion for curious graphs of type at most~$2$.
Applying Lemmas~\ref{lem:slice-cw}, \ref{lem:slice-wqo} and~\ref{lem:cur-type23-slice-type12} again, we obtain the same conclusion for curious graphs of type at most~$3$, that is, all curious graphs.
This completes the proof.\qed
\end{proof}

\section{Applications of Our Technique}\label{s-main}

In this section we show that the class of $(K_3,P_1+\nobreak P_5)$-free graphs and the class of $(K_3,P_2+\nobreak P_4)$-free graphs have bounded clique-width and are well-quasi-ordered by the labelled induced subgraph relation.
To do this, we first prove two structural lemmas.
These lemmas consider the case where a graph in one of these classes contains an induced subgraph isomorphic to~$C_5$ and show how to use a bounded number of vertex deletions and bipartite complementations to transform the graph into a disjoint union of curious graphs and $3$-uniform graphs in the respective graph class.
We then use these lemmas to prove Theorem~\ref{thm:main}, which is our main result.

The first of the two lemmas is implicit in the proofs of~\cite[Lemma~9 and Theorem~3]{DDP15}, but without an explicit upper bound on the number of operations used and the number of obtained curious graphs.
For completeness, we give a direct proof and provide such explicit bounds.
Note that the bounds on the number of vertex deletions, bipartite complementations and curious graphs in both lemmas are not tight, but any upper bound on these numbers will be sufficient for our purposes.

\begin{lemma}\label{lem:K3P5+P1C5}
Given any connected $(K_3,P_1+\nobreak P_5)$-free graph~$G$ that contains an induced~$C_5$, we can apply at most~5 vertex deletions and at most~31 bipartite complementation operations to obtain a graph~$H$ that is the disjoint union of~11 $(K_3,P_1+\nobreak P_5)$-free curious graphs.
\end{lemma}

\begin{proof}
\setcounter{ctrclaim}{0}
Let~$G$ be a connected $(K_3,P_1+\nobreak P_5)$-free graph that contains an induced cycle~$C$ on the vertices $v_1,\ldots,v_5$, listed in order along the cycle.
To aid notation, for the remainder of the proof subscripts on vertices and on sets should be interpreted modulo~$5$.
We will show how to use vertex deletions and bipartite complementations to partition the graph into a disjoint union of curious graphs.
At the end of the proof, we will verify the number of operations used.

Since~$G$ is $K_3$-free, every vertex not on the cycle~$C$ has at most two neighbours on~$C$ and if it does have two neighbours on~$C$, then these neighbours must be non-consecutive vertices of the cycle.
We may therefore partition the vertices in $V(G) \setminus V(C)$ into eleven sets as follows:
\begin{itemize}
\item $U$ is the set of vertices anti-complete to~$C$,
\item for $i \in \{1,\ldots,5\}$, $W_i$ is the set of vertices whose unique neighbour on~$C$ is~$v_i$ and
\item for $i \in \{1,\ldots,5\}$, $V_i$ is the set of vertices adjacent to~$v_{i-1}$ and~$v_{i+1}$ and non-adjacent to the rest of~$C$.
\end{itemize}

We prove a series of claims.

\clm{\label{clm4:U-empty}We may assume that $U = \emptyset$.}
Suppose, for contradiction, that there are two vertices $u,u' \in U$ that do not have the same neighbourhood in some set~$V_i$ or~$W_i$.
Say $v \in V_i \cup W_{i+1}$ is adjacent to~$u$, but not~$u'$ for some $i \in \{1,\ldots,5\}$.
Note that~$v$ is adjacent to~$v_{i+1}$, but non-adjacent to $v_i,v_{i+2}$ and~$v_{i+3}$.
Then $G[v_{i+3},u',u,v,v_{i+1},v_i]$ is a $P_1+\nobreak P_5$ if~$u$ and~$u'$ are adjacent and $G[u',u,v,v_{i+1},v_{i+2},v_{i+3}]$ is a $P_1+\nobreak P_5$ if they are not.
This contradiction means that every vertex in~$U$ has the same neighbourhood in every set~$V_i$ and every set~$W_i$.
Since~$G$ is connected, there must be a vertex~$v$ in some~$V_i$ or~$W_i$ that is adjacent to every vertex of~$U$.
Since~$G$ is $K_3$-free, $U$ must therefore be an independent set.
Applying a bipartite complementation between~$U$ and the vertices adjacent to the vertices of~$U$ disconnects~$G[U]$ from the rest of the graph.
Since~$U$ is independent, $G[U]$ is a curious graph (with two of the three partition classes empty).
This completes the proof of Claim~\ref{clm4:U-empty}.

\clm{\label{clm4:V_i-1WiV_i+1-indep}For $i \in \{1,\ldots,5\}$, $V_{i-1} \cup W_i \cup V_{i+1}$ is an independent set.}
Indeed, if $x,y \in V_{i-1} \cup W_i \cup V_{i+1}$ are adjacent then $G[x,y,v_i]$ is a~$K_3$, a contradiction.
This completes the proof of Claim~\ref{clm3:V_i-1WiV_i+1-indep}.

\clm{\label{clm4:Wi-Wi+1-comp}For $i \in \{1,\ldots,5\}$, $W_i$ is complete to $W_{i-1} \cup W_{i+1}$.}
Indeed, if $x \in W_{i-1}$ is non-adjacent to $y \in W_i$ then $G[x,y,v_i,v_{i+1},v_{i+2},v_{i+3}]$ is a $P_1+\nobreak P_5$, a contradiction.
Claim~\ref{clm4:Wi-Wi+1-comp} follows by symmetry.

\clm{\label{clm4:Vi-Wi+2-triv}For $i \in \{1,\ldots,5\}$, $W_i$ is complete or anti-complete to~$W_{i+2}$.}
Suppose, for contradiction, that $x \in W_{i+2}$ has a neighbour $y \in W_i$ and a non-neighbour $y' \in W_i$.
By Claim~\ref{clm4:V_i-1WiV_i+1-indep}, $y$ is non-adjacent to~$y'$.
Then $G[y',y,x,v_{i+2},v_{i+3},v_{i+4}]$ is a $P_1+\nobreak P_5$, a contradiction.
Claim~\ref{clm4:Vi-Wi+2-triv} follows by symmetry.

\clm{\label{clm4:Vi-Wi-triv}For $i \in \{1,\ldots,5\}$, $V_i$ is either complete or anti-complete to~$W_i$.}
Suppose, for contradiction, that $x \in V_i$ has a neighbour $y \in W_i$ and a non-neighbour $y' \in W_i$.
By Claim~\ref{clm4:V_i-1WiV_i+1-indep}, $y$ is non-adjacent to~$y'$.
Then $G[y',y,x,v_{i+1},v_{i+2},v_{i+3}]$ is a $P_1+\nobreak P_5$, a contradiction.
Therefore every vertex in~$V_i$ is complete or anti-complete to~$W_i$.
Now suppose, for contradiction, that $y \in W_i$ has a neighbour $x \in V_i$ and a non-neighbour $x' \in V_i$.
By Claim~\ref{clm4:V_i-1WiV_i+1-indep}, $x$ is non-adjacent to~$x'$.
Then $G[v_{i+3},v_i,y,x,v_{i+1},x']$ is a $P_1+\nobreak P_5$, a contradiction.
Therefore every vertex in~$W_i$ is complete or anti-complete to~$V_i$.
This completes the proof of Claim~\ref{clm4:Vi-Wi-triv}.

\clm{\label{clm4:Vi-dominates-Vi-1orVi+1}For $i \in \{1,\ldots,5\}$, every vertex of~$V_i$ is complete to either~$V_{i-1}$ or~$V_{i+1}$.}
Suppose, for contradiction, that $x \in V_i$ is non-adjacent to $y \in V_{i-1}$ and $z \in V_{i+1}$.
By Claim~\ref{clm4:V_i-1WiV_i+1-indep}, $y$ is non-adjacent to~$z$.
Then $G[y,v_{i-1},x,v_{i+1},v_{i+2},z]$ is a $P_1+\nobreak P_5$, a contradiction.
This completes the proof of Claim~\ref{clm4:Vi-dominates-Vi-1orVi+1}.

\bigskip
\noindent
By Claim~\ref{clm4:Vi-dominates-Vi-1orVi+1}, we can partition each set~$V_i$ into three (possibly empty) subsets as follows:
\begin{itemize}
\item $V^0_i$ the set of vertices in~$V_i$ that dominate both~$V_{i-1}$ and~$V_{i+1}$,
\item $V^-_i$ the set of vertices in $V_i\setminus V^0_i$ that dominate~$V_{i-1}$ (and so have non-neighbours in~$V_{i+1}$),
\item $V^+_i$ the set of vertices in $V_i\setminus V^0_i$ that dominate~$V_{i+1}$ (and so have non-neighbours in~$V_{i-1}$).
\end{itemize}

\noindent
By definition of this partition, if $x\in V_i$ is non-adjacent to $y\in V_{i+1}$, then $x \in V^-_i$ and $y \in V^+_{i+1}$.
Moreover, {\em every} vertex of~$V^-_i$ has a non-neighbour in~$V^+_{i+1}$ and vice versa. 
Thus, the vertices of $V_1\cup\cdots\cup V_5$ are partitioned into 15 subsets (some or all of which may be empty).
Note that for $i \in \{1,\ldots,5\}$ the sets $V_i^0$, $V^-_i$, $V^+_i$, $V_{i+2}^0$, $V^-_{i+2}$ and~$V^+_{i+2}$ are pairwise anti-complete by Claim~\ref{clm4:V_i-1WiV_i+1-indep}.
Furthermore, for $i \in \{1,\ldots,5\}$, $V_i^0$ and~$V^+_i$ are complete to $V_{i+1}^0$, $V^-_{i+1}$ and~$V^+_{i+1}$, and $V^-_i$ is complete to $V_{i+1}^0$ and~$V^-_{i+1}$.
Therefore nearly every pair of these subsets is either complete or anti-complete to each other.
The only possible exceptions are the five disjoint pairs of the form $\{V^-_i,V^+_{i+1}\}$.

Next, for each $i \in \{1,\ldots,5\}$ we analyse the edges between~$W_i$ and the subsets of~$V_{i-2}$ and~$V_{i+2}$.

\newpage
\clm{\label{clm4:Wi-anti-Vi-2-Vi+2+}For $i \in \{1,\ldots,5\}$, $W_i$ is anti-complete to $V_{i-2}^-\cup V_{i+2}^+$.}
Suppose, for contradiction, that $x \in W_i$ has a neighbour $y \in V_{i+2}^+$.
By definition of~$V_{i+2}^+$, $y$ must have a non-neighbour $z \in V_{i+1}$.
By Claim~\ref{clm4:V_i-1WiV_i+1-indep}, $z$ is non-adjacent to~$x$.
Then $G[v_{i-1},x,y,v_{i+1},v_{i+2},z]$ is a $P_1+\nobreak P_5$, a contradiction.
Claim~\ref{clm4:Wi-anti-Vi-2-Vi+2+} follows by symmetry.

\clm{\label{clm4:W_iVi+2Vi+3-no-3P1}For $i \in \{1,\ldots,5\}$, if $x\in W_i$, $y\in V_{i+2}$ and $z\in V_{i-2}$ then $G[x,y,z]$ is not a~$3P_1$.}
Indeed, if $x\in W_i$, $y\in V_{i+2}$ and $z\in V_{i-2}$ and $G[x,y,z]$ is a~$3P_1$ then $G[z,x,v_i,v_{i+1},\allowbreak y,v_{i-2}]$ is a $P_1+\nobreak P_5$, a contradiction.
This completes the proof of Claim~\ref{clm4:W_iVi+2Vi+3-no-3P1}.

\clmnonewline{\label{clm4:x-in-Wi-nbr-in-W0+_i+2-then-anti-Vi+3-comp-V-_i+2}For $i \in \{1,\ldots,5\}$:
\begin{enumerate}[(i)]
\item If $x\in W_i$ has a neighbour in~$V^0_{i+2}$, then~$x$ is anti-complete to~$V_{i+3}$ and complete to~$V^-_{i+2}$.
\item If $x\in W_i$ has a neighbour in~$V^0_{i+3}$, then~$x$ is anti-complete to~$V_{i+2}$ and complete to~$V^+_{i+3}$.
\end{enumerate}
}
\noindent
Suppose $x\in W_i$ has a neighbour $y \in V^0_{i+2}$.
If $z \in V_{i+3}$ then~$y$ is adjacent to~$z$, so~$x$ must be non-adjacent to~$z$, otherwise $G[x,y,z]$ would be a~$K_3$.
Therefore~$x$ is anti-complete to~$V_{i+3}$.
If $y' \in V^-_{i+2}$ then~$y'$ has a non-neighbour $z' \in V_{i+3}$.
Note that~$z'$ is non-adjacent to~$x$.
Now~$x$ must be adjacent to~$y'$, otherwise $G[x,y',z']$ would be a~$3P_1$, contradicting Claim~\ref{clm4:W_iVi+2Vi+3-no-3P1}.
It follows that~$x$ is complete to~$V^-_{i+2}$.
Claim~\ref{clm4:x-in-Wi-nbr-in-W0+_i+2-then-anti-Vi+3-comp-V-_i+2} follows by symmetry.

\medskip
\noindent
Recall that~$W_i$ is anti-complete to $V_{i-2}^-\cup V_{i+2}^+$ by Claim~\ref{clm4:Wi-anti-Vi-2-Vi+2+}.
By Claim~\ref{clm4:x-in-Wi-nbr-in-W0+_i+2-then-anti-Vi+3-comp-V-_i+2} we can partition~$W_i$ into three (possibly empty) subsets as follows:
\begin{itemize}
\item $W^2_i$ the set of vertices in~$W_i$ that have neighbours in~$V^0_{i+2}$ (and are therefore anti-complete to~$V_{i+3}$ and complete to~$V^-_{i+2}$),
\item $W^3_i$ the set of vertices in~$W_i$ that have neighbours in~$V^0_{i+3}$ (and are therefore anti-complete to~$V_{i+2}$ and complete to~$V^+_{i+3}$) and
\item $W^0_i$ the set of vertices in $W_i\setminus (W^2_i\cup W^3_i)$ (which are therefore anti-complete to $V^0_{i+2} \cup V^0_{i+3}$).
\end{itemize}

\clm{\label{clm4:Wi0Vi+Vi--empty}For $i \in \{1,\ldots,5\}$, we may assume that $W_i^0 \cup V_{i-2}^+ \cup V_{i+2}^-=\emptyset$.}
We first show how to remove all edges from vertices in $G[W_i^0 \cup V_{i-2}^+ \cup V_{i+2}^-]$ to the rest of~$G$ by using three bipartite complementations.

By Claim~\ref{clm4:V_i-1WiV_i+1-indep}, $W^0_i$ is anti-complete to~$W^2_i$, $W^3_i$, $V_{i-1}$ and~$V_{i+1}$.
By Claim~\ref{clm4:Wi-anti-Vi-2-Vi+2+}, $W^0_i$ is anti-complete to~$V_{i+3}^-$ and~$V_{i+2}^+$.
By definition, $W^0_i$ is anti-complete to~$V^0_{i+2}$ and~$V^0_{i+3}$.
By Claim~\ref{clm4:Wi-Wi+1-comp}, $W^0_i$ is complete to $W_{i-1}$ and~$W_{i+1}$.
By Claim~\ref{clm4:Vi-Wi-triv}, $W^0_i$ is either complete or anti-complete to~$V_i$.
By Claim~\ref{clm4:Vi-Wi+2-triv}, $W^0_i$ is either complete or anti-complete to~$W_{i+2}$ and either complete or anti-complete to~$W_{i+3}$.
By definition of~$W_i$, $v_i$ is complete to~$W_i$ and $C \setminus \{v_i\}$ is anti-complete to~$W_i$.
Therefore, by applying at most one bipartite complementation, we can remove all edges with one end-vertex in~$W^0_i$ and the other end-vertex outside $W_i^0 \cup V_{i-2}^+ \cup V_{i+2}^-$.

By Claim~\ref{clm4:V_i-1WiV_i+1-indep}, $V_{i+2}^-$ is anti-complete to $V_i$, $V_{i-1}$, $V_{i+2}^+$, $V_{i+2}^0$, $W_{i+1}$ and~$W_{i+3}$.
By Claim~\ref{clm4:Vi-Wi-triv}, $V_{i+2}^-$ is either complete or anti-complete to~$W_{i+2}$.
By Claim~\ref{clm4:Wi-anti-Vi-2-Vi+2+}, $V_{i+2}^-$ is anti-complete to~$W_{i-1}$.
By definition, $V_{i+2}^-$ is complete to $V_{i+1}$.
By definition of~$V_{i+3}^0$ and~$V_{i+3}^-$, $V_{i+2}^-$ is complete to~$V_{i+3}^0$ and~$V_{i+3}^-$.
By Claim~\ref{clm4:x-in-Wi-nbr-in-W0+_i+2-then-anti-Vi+3-comp-V-_i+2}, $V_{i+2}^-$ is complete to~$W^2_i$ and anti-complete to~$W^3_i$.
By definition of~$V_{i+2}$, $\{v_{i+1},v_{i+3}\}$ is complete to~$V_{i+2}^-$ and $C \setminus \{v_{i+1},v_{i+3}\}$ is anti-complete to~$V_{i+2}^-$.
Therefore, by applying at most one bipartite complementation, we can remove all edges with one end-vertex in~$V_{i+2}^-$ and the other end-vertex outside $W_i^0 \cup V_{i-2}^+ \cup V_{i+2}^-$.
Symmetrically, by applying at most one bipartite complementation, we can remove all edges with one end-vertex in~$V_{i-2}^+$ and the other end-vertex outside $W_i^0 \cup V_{i-2}^+ \cup V_{i+2}^-$.

Thus we can apply three bipartite complementations to disconnect $G[W_i^0 \cup V_{i-2}^+ \cup V_{i+2}^-]$ from the rest of the graph.
By Claim~\ref{clm4:V_i-1WiV_i+1-indep}, $W_i^0$, $V_{i-2}^+$ and~$V_{i+2}^-$ are independent sets.
Therefore, by Claim~\ref{clm4:W_iVi+2Vi+3-no-3P1}, $G[W_i^0 \cup V_{i-2}^+ \cup V_{i+2}^-]$ is a curious graph.
This completes the proof of Claim~\ref{clm4:Wi0Vi+Vi--empty}.

\bigskip
\noindent
By the Claim~\ref{clm4:Wi0Vi+Vi--empty}, the only sets that remain non-empty are ones of form~$W_i^3$, $W_i^2$ or~$V_i^0$.
We prove the following claim.

\clm{\label{clm4:Wi3Wi2Vi0-empty}For $i \in \{1,\ldots,5\}$, we may assume that $W_i^3 \cup W_{i+1}^2\cup V_{i-2}^0=\emptyset$.}
We first show how to remove all edges from vertices in $G[W_i^3 \cup W_{i+1}^2\cup V_{i-2}^0]$ to the rest of the graph by using finitely many bipartite complementations.

By Claim~\ref{clm4:V_i-1WiV_i+1-indep}, $W_i^3$ is anti-complete to~$W_i^2$, $V_{i-1}^0$ and~$V_{i+1}^0$.
By Claim~\ref{clm4:Wi-Wi+1-comp}, $W_i^3$ is complete to~$W_{i-1}$ and~$W_{i+1}^3$.
By Claim~\ref{clm4:Vi-Wi+2-triv}, $W_i^3$ is either complete or anti-complete to~$W_{i+2}$ and either complete or anti-complete to~$W_{i-2}$.
By Claim~\ref{clm4:Vi-Wi-triv}, $W_i^3$ is either complete or anti-complete to~$V_i^0$.
By Claim~\ref{clm4:x-in-Wi-nbr-in-W0+_i+2-then-anti-Vi+3-comp-V-_i+2}, $W_i^3$ is anti-complete to~$V_{i+2}^0$.
By definition of~$W_i$, $v_i$ is complete to~$W_i^3$ and $C \setminus \{v_i\}$ is anti-complete to~$W_i^3$.
Therefore, by applying at most one bipartite complementation, we can remove all edges with one end-vertex in~$W_i^3$ and the other end-vertex outside $W_i^3 \cup W_{i+1}^2\cup V_{i-2}^0$.
Symmetrically, by applying at most one bipartite complementation, we can remove all edges with one end-vertex in~$W_{i+1}^2$ and the other end-vertex outside $W_i^3 \cup W_{i+1}^2\cup V_{i-2}^0$.

By Claim~\ref{clm4:V_i-1WiV_i+1-indep}, $V_{i-2}^0$ anti-complete to~$V_i$, $W_{i-1}$, $W_{i+2}$, $V_{i+1}$.
By Claim~\ref{clm4:Vi-Wi-triv}, $V_{i-2}^0$ is either complete or anti-complete to~$W_{i-2}$.
By definition, $V_{i-2}^0$ is complete to~$V_{i-1}^0$ and~$V_{i-3}^0$.
By Claim~\ref{clm4:x-in-Wi-nbr-in-W0+_i+2-then-anti-Vi+3-comp-V-_i+2}, $V_{i-2}^0$ is anti-complete to~$W_i^2$ and~$W_{i+1}^3$.
By definition of~$V_{i-2}$, $\{v_{i-1},v_{i-3}\}$ is complete to~$V_{i-2}^0$ and $C \setminus \{v_{i-1},v_{i-3}\}$ is anti-complete to~$V_{i-2}^0$.
Therefore, by applying at most one bipartite complementation, we can remove all edges with one end-vertex in~$V_{i-2}^0$ and the other end-vertex outside $W_i^3 \cup W_{i+1}^2\cup V_{i-2}^0$.

Thus we can apply three bipartite complementations to disconnect $G[W_i^3 \cup W_{i+1}^2\cup V_{i-2}^0]$ from the rest of the graph.
By Claim~\ref{clm4:V_i-1WiV_i+1-indep}, $W_i^3$, $W_{i+1}^2$ and~$V_{i-2}^0$ are independent sets.
By Claim~\ref{clm4:Wi-Wi+1-comp}, $W_i^3$ is complete to~$W_{i+1}^2$, so $G[W_i^3 \cup W_{i+1}^2\cup V_{i-2}^0]$ is a curious graph.
This completes the proof of Claim~\ref{clm4:Wi3Wi2Vi0-empty}.

\medskip
\noindent
By Claim~\ref{clm4:Wi3Wi2Vi0-empty}, we may assume that~$G$ only contains the five vertices in the cycle~$C$.
We delete these vertices.

\bigskip
\noindent
To complete the proof, it remains to verify the number of operations applied and the number of obtained curious graphs.
In Claim~\ref{clm4:U-empty}, we apply one bipartite complementation and obtain one curious graph.
In Claim~\ref{clm4:Wi0Vi+Vi--empty}, for each $i \in \{1,\ldots,5\}$, we apply three bipartite complementations and obtain one curious graph.
In Claim~\ref{clm4:Wi3Wi2Vi0-empty}, for each $i \in \{1,\ldots,5\}$, we apply three bipartite complementations and obtain one curious graph.
Finally, we apply five vertex deletions to delete the original cycle~$C$ from the graph.
This leads to a total of five vertex deletions, $1+(5\times 3)+(5 \times 3)=31$ bipartite complementations and $1+5+5=11$ obtained curious graphs.\qed
\end{proof}

We now prove our second structural lemma.

\begin{lemma}\label{lem:K3P4+P2C5}
Given any prime $(K_3,P_2+\nobreak P_4)$-free graph~$G$ that contains an induced~$C_5$, we can apply at most 2570 vertex deletions and at most 459 bipartite complementation operations to obtain a graph~$H$ that is the disjoint union of at most 19 $(K_3,P_2+\nobreak P_4)$-free curious graphs and at most one $3$-uniform graph.
\end{lemma}

\begin{proof}
\setcounter{ctrclaim}{0}
Let~$G$ be a prime $(K_3,P_2+\nobreak P_4)$-free graph that contains an induced cycle~$C$ on the vertices $v_1,\ldots,v_5$, listed in order along the cycle.
To aid notation, for the remainder of the proof subscripts on vertices and on sets should be interpreted modulo~$5$.
We will show how to use vertex deletions and bipartite complementations to partition the graph into a disjoint union of curious graphs and $3$-uniform graphs.
At the end of the proof, we will verify the number of operations used.

Since~$G$ is $K_3$-free, every vertex not on the cycle~$C$ has at most two neighbours on~$C$ and if it does have two neighbours on~$C$, then these neighbours must be non-consecutive vertices of the cycle.
We may therefore partition the vertices in $V(G) \setminus V(C)$ into eleven sets as follows:
\begin{itemize}
\item $U$ is the set of vertices anti-complete to~$C$,
\item for $i \in \{1,\ldots,5\}$, $W_i$ is the set of vertices whose unique neighbour on~$C$ is~$v_i$ and
\item for $i \in \{1,\ldots,5\}$, $V_i$ is the set of vertices adjacent to~$v_{i-1}$ and~$v_{i+1}$ and non-adjacent to the rest of~$C$.
\end{itemize}

\noindent
We start by showing how to use a bounded number of vertex deletions and bipartite complementations to disconnect a $3$-uniform or bipartite induced subgraph containing~$U$ from the rest of the graph.
This will enable us to assume that~$U=\emptyset$.
In aid of this, we prove a series of claims.

\clm{\label{clm3:UWi-indep}For $i \in \{1,\ldots,5\}$, $U \cup W_i$ is an independent set.}
Indeed, if $x,y \in U \cup W_i$ are adjacent then $G[x,y,v_{i+1},v_{i+2},v_{i+3},v_{i+4}]$ is a $P_2+\nobreak P_4$, a contradiction.
This completes the proof of Claim~\ref{clm3:UWi-indep}.

\clm{\label{clm3:V_i-1WiV_i+1-indep}For $i \in \{1,\ldots,5\}$, $V_{i-1} \cup W_i \cup V_{i+1}$ is an independent set.}
Indeed, if $x,y \in V_{i-1} \cup W_i \cup V_{i+1}$ are adjacent then $G[x,y,v_i]$ is a~$K_3$, a contradiction.
This completes the proof of Claim~\ref{clm3:V_i-1WiV_i+1-indep}.

\clm{\label{clm3:UVi-P4-free}For $i \in \{1,\ldots,5\}$, $G[U \cup V_i]$ is $P_4$-free.}
Indeed, if $x_1,x_2,x_3,x_4 \in U \cup V_i$ induce a~$P_4$ then $G[v_{i+2},v_{i+3},x_1,x_2,x_3,x_4]$ would be a $P_2+\nobreak P_4$, a contradiction.
This completes the proof of Claim~\ref{clm3:UVi-P4-free}.

\begin{sloppypar}
\clm{\label{clm3:x-in-Vi-triv-to-Wi}For $i \in \{1,\ldots,5\}$, every vertex in~$V_i$ is either complete or anti-complete to~$W_i$.}
Indeed, suppose for contradiction that $x \in V_i$ has a neighbour $y \in W_i$ and a non-neighbour $y' \in W_i$.
By Claim~\ref{clm3:UWi-indep}, $y$ is non-adjacent to~$y'$.
It follows that $G[v_{i+2},v_{i+3},x,y,v_i,y']$ is a $P_2+\nobreak P_4$, a contradiction.
This completes the proof of Claim~\ref{clm3:x-in-Vi-triv-to-Wi}.
\end{sloppypar}

\medskip
\noindent
For $i \in \{1,\ldots,5\}$, let~$V_i^*$ denote the set of vertices in~$V_i$ that have neighbours in~$U$ and let $V^0_i=V_i\setminus V^*_i$.
Let $V^*=\bigcup V^*_i$.
We will show how to use five bipartite complementations to separate $G[U \cup V^*]$ from the rest of the graph.
Note that by definition and by Claim~\ref{clm3:UWi-indep}, no vertex of~$U$ has a neighbour outside~$V^*$.
It is therefore sufficient to show how to disconnect~$V^*$ from the vertices outside $U\cup V^*$.

In fact, for vertices in~$V^*_i$ we can prove a stronger version of Claim~\ref{clm3:x-in-Vi-triv-to-Wi}.

\begin{sloppypar}
\clm{\label{clm3:Vistar-anti-Wi}For $i \in \{1,\ldots,5\}$, $V_i^*$ is anti-complete to~$W_i$.}
Indeed, suppose for contradiction that $x \in V_i^*$ is adjacent to $y \in W_i$ and let $z \in U$ be a neighbour of~$x$.
By Claim~\ref{clm3:UWi-indep}, $y$ is non-adjacent to~$z$.
If follows that $G[v_{i+2},v_{i+3},v_i,y,x,z]$ is a $P_2+\nobreak P_4$, a contradiction.
This completes the proof of Claim~\ref{clm3:Vistar-anti-Wi}.
\end{sloppypar}

\clm{\label{clm3:V_i-U-nbr-V_i=1Vi+1-distinguishes}For $i \in \{1,\ldots,5\}$, if $x \in V^*_i$ is adjacent to $y \in U$ then every vertex $z \in V_{i-1} \cup V_{i+1}$ has exactly one neighbour in $\{x,y\}$.}
Indeed, suppose $x \in V^*_i$, $y \in U$ and $z \in V_{i-1}$ with~$x$ adjacent to~$y$.
If~$z$ is adjacent to both~$x$ and~$y$ then $G[x,y,z]$ is a~$K_3$, a contradiction.
If~$z$ is non-adjacent to both~$x$ and~$y$ then $G[x,y,v_i,z,v_{i+3},v_{i+2}]$ is a $P_2+\nobreak P_4$, a contradiction.
Claim~\ref{clm3:V_i-U-nbr-V_i=1Vi+1-distinguishes} follows by symmetry.

\clm{\label{clm3:Vistar-compVi-10Vi+10}For $i \in \{1,\ldots,5\}$, $V^*_i$ is complete to $V^0_{i-1} \cup V^0_{i+1}$.}
This follows from Claim~\ref{clm3:V_i-U-nbr-V_i=1Vi+1-distinguishes} by definition of~$V^*_i$ and $V^0_{i-1} \cup V^0_{i+1}$.
This completes the proof of Claim~\ref{clm3:Vistar-compVi-10Vi+10}.

\clm{\label{clm3:Vistar-compWi-2Wi+2}For $i \in \{1,\ldots,5\}$, $V^*_i$ is complete to $W_{i-2} \cup W_{i+2}$.}
Suppose, for contradiction, that $x \in V^*_i$ is non-adjacent to a vertex $y \in W_{i+2}$ and let $z \in U$ be a neighbour of~$x$.
By Claim~\ref{clm3:UWi-indep}, $y$ is non-adjacent to~$z$.
Therefore $G[y,v_{i+2},v_i,v_{i-1},x,z]$ is a $P_2+\nobreak P_4$, a contradiction.
Claim~\ref{clm3:Vistar-compWi-2Wi+2} follows by symmetry.

\clm{\label{clm3:VstarU-seperable}For $i \in \{1,\ldots,5\}$, $V_i^*$ is complete to $V^0_{i-1} \cup V^0_{i+1} \cup W_{i-2} \cup W_{i+2} \cup \{v_{i-1}, v_{i+1}\}$ and anti-complete to every vertex of $G \setminus (V^* \cup U)$ that is not in $V^0_{i-1} \cup V^0_{i+1} \cup W_{i-2} \cup W_{i+2} \cup \{v_{i-1}, v_{i+1}\}$.
Furthermore, no vertex in~$U$ has a neighbour in $G \setminus (V^* \cup U)$.}
By definition of~$V_i^*$, every vertex in~$V_i^*$ is complete to $\{v_{i-1}, v_{i+1}\}$ and anti-complete to $C \setminus \{v_{i-1}, v_{i+1}\}$.
By Claim~\ref{clm3:V_i-1WiV_i+1-indep}, $V^*_i$ is anti-complete to $W_{i-1} \cup W_{i+1} \cup V^0_{i-2} \cup V^0_{i+2} \cup V^0_i$.
By Claim~\ref{clm3:Vistar-compVi-10Vi+10}, $V^*_i$ is complete to $V^0_{i-1} \cup V^0_{i+1}$.
By Claim~\ref{clm3:Vistar-anti-Wi}, $V^*_i$ is anti-complete to~$W_i$.
By Claim~\ref{clm3:Vistar-compWi-2Wi+2}, $V^*_i$ is complete to $W_{i-2} \cup W_{i+2}$.
By definition of~$U$, every vertex in~$U$ is anti-complete to~$C$.
By Claim~\ref{clm3:UWi-indep}, $U$ is anti-complete to~$W_j$ for all $j \in \{1,\ldots,5\}$ and by definition of~$V_j^0$, $U$ is non-adjacent to~$V_j^0$ for all $j \in \{1,\ldots,5\}$.
This completes the proof of Claim~\ref{clm3:VstarU-seperable}.

\medskip
\noindent
Recall that at the start of the proof we assumed that the graph~$G$ is prime.
By Claim~\ref{clm3:VstarU-seperable} if for each $i \in \{1,\ldots,5\}$ we apply a bipartite complementation between~$V_i^*$ and $V^0_{i-1} \cup V^0_{i+1} \cup W_{i-2} \cup W_{i+2} \cup \{v_{i-1}, v_{i+1}\}$ then this would remove all edges between $V^* \cup U$ and $V(G) \setminus (V^* \cup U)$ and leave all other edges unchanged.
However, after doing this, the result might be a graph that is not prime.
As such, we will not use Claim~\ref{clm3:VstarU-seperable} to disconnect $G[V^* \cup U]$ from the rest of the graph at this stage, but wait until later in the proof, when we no longer require the property that~$G$ is prime.
Next, we analyse the structure of $G[V^* \cup U]$.
Note that $V_1,\ldots,V_5$ and~$U$ are independent sets by Claims~\ref{clm3:V_i-1WiV_i+1-indep} and~\ref{clm3:UWi-indep}.
By Claim~\ref{clm3:UVi-P4-free}, this implies that every component of $G[U \cup V_i]$ is a complete bipartite graph.
Fix $i \in \{1,\ldots,5\}$ and consider a component~$C'$ of $G[U \cup V_i^*]$ containing at least one vertex of~$V^*_i$.
Let $T= V(C') \cap V^*_i$ and $S = V(C') \cap U$.
Note that since every vertex of~$V^*_i$ contains at least one neighbour in~$U$ if follows that $|S| \geq 1$.
We prove a series of claims.

\clm{\label{clm3:T-small}$|T| = 1$.}
We first show that~$T$ is a module in~$G$.
Since $T \subseteq V_i^*$, if a vertex~$v$ in~$G$ can distinguish two vertices in~$T$ then $v \in V^* \cup U$ by Claim~\ref{clm3:VstarU-seperable}.
Since~$U$ and~$V_i^*$ are independent sets by Claims~\ref{clm3:UWi-indep} and~\ref{clm3:V_i-1WiV_i+1-indep}, respectively and every component of $G[U \cup V_i^*]$ is a complete bipartite graph by Claim~\ref{clm3:UVi-P4-free}, no vertex of $(U \cup V_i^*) \setminus T$ can distinguish two vertices in~$T$.
By Claim~\ref{clm3:V_i-1WiV_i+1-indep}, no vertex of $V_{i-2}^* \cup V_{i+2}^*$ can distinguish two vertices in~$T$.
Because $G[S \cup T]$ is a complete bipartite graph, Claim~\ref{clm3:V_i-U-nbr-V_i=1Vi+1-distinguishes} implies that every vertex of $V_{i-1}^*\cup V_{i+1}^*$ is either complete to~$S$ and anti-complete to~$T$ or anti-complete to~$S$ and complete to~$T$, so no vertex of $V_{i-1}^*\cup V_{i+1}^*$ can distinguish two vertices in~$T$.
It follows that~$T$ is a module in~$G$.
Since~$G$ is prime, we conclude that $|T|=1$.
This completes the proof of Claim~\ref{clm3:T-small}.

\clm{\label{clm3:Vistar-large-S-small}If $|V^*_i|>1$, then $|S|=1$.}
We show that if $|V^*_i|>1$ then~$S$ is a module in~$G$.
Suppose, for contradiction, that~$S$ is not a module in~$G$.
Since $S \subseteq U$, if a vertex~$v$ in~$G$ can distinguish two vertices in~$S$ then $v \in V^* \cup U$ by Claim~\ref{clm3:VstarU-seperable}.
Since~$U$ and~$V_i^*$ are independent sets by Claims~\ref{clm3:UWi-indep} and~\ref{clm3:V_i-1WiV_i+1-indep}, respectively and every component of $G[U \cup V_i^*]$ is a complete bipartite graph by Claim~\ref{clm3:UVi-P4-free}, no vertex of $(U \cup V_i^*) \setminus S$ can distinguish two vertices in~$S$.
Because $G[S \cup T]$ is a complete bipartite graph, Claim~\ref{clm3:V_i-U-nbr-V_i=1Vi+1-distinguishes} implies that every vertex of $V_{i-1}^*\cup V_{i+1}^*$ is either complete to~$S$ and anti-complete to~$T$ or anti-complete to~$S$ and complete to~$T$, so no vertex of $V_{i-1}^*\cup V_{i+1}^*$ can distinguish two vertices in~$S$.
Therefore, there must be a vertex in $V_{i-2} \cup V_{i+2}$ that distinguishes two vertices of~$S$.
Without loss of generality, assume that there is a vertex $x \in V_{i+2}$ that is adjacent to $y \in S$ and non-adjacent to $y' \in S$.
Let~$z$ be a vertex of~$T$, i.e. a vertex of~$V^*_i$ that is adjacent to both~$y$ and~$y'$.
Note that $|V^*_i|>1$ and any two vertices of~$V^*_i$ belong to different connected components of $G[V^*_i\cup U]$ by Claim~\ref{clm3:T-small}.
This means that there must be a vertex $z' \in V^*_i \setminus \{z\}$ and it must have a neighbour $y'' \in U \setminus S$.
Note that~$x$ is non-adjacent to~$z$ and~$z'$ by Claim~\ref{clm3:V_i-1WiV_i+1-indep}.
It follows that $G[y',z,v_{i+2},v_{i+3},x,y'']$ or $G[y'',z',x,y,z,y']$ is a $P_2+\nobreak P_4$ if~$x$ is adjacent or non-adjacent to~$y''$, respectively.
This contradiction implies that~$S$ is indeed a module.
Since~$G$ is prime, we conclude that $|S|=1$.
This completes the proof of Claim~\ref{clm3:Vistar-large-S-small}.

\clm{\label{clm3:Vistar-large-V_istar-2Vi+2star-small}If $|V^*_i|>1$, then $|V^*_{i-2}|\leq 1$ and $|V^*_{i+2}|\leq 1$.}
Suppose, for contradiction, that~$V^*_i$ and~$V^*_{i+2}$ each contain at least two vertices.
Then by Claims~\ref{clm3:T-small} and~\ref{clm3:Vistar-large-S-small}, every connected component of $G[V_i^* \cup U]$ containing a vertex of~$V_i^*$ consists of a single edge and since~$V^*_i$ contains at least two vertices, there must be at least two such components.
Similarly, there must be two such components in $G[V_{i+2}^* \cup U]$.
Therefore we can find $x \in V^*_i$ adjacent to $y \in U$ and $y' \in U \setminus \{y\}$ adjacent to $z \in V^*_{i+2}$.
By Claim~\ref{clm3:V_i-1WiV_i+1-indep}, $x$ is non-adjacent to~$z$, so $G[x,y,y',z,v_{i+3},v_{i+2}]$ is a $P_2+\nobreak P_4$, a contradiction.
Claim~\ref{clm3:Vistar-large-V_istar-2Vi+2star-small} follows by symmetry.

\clm{\label{clm3:VstarcupU-done}There is a set $V^{**} \subseteq V^*$ with $|V^{**}| \leq 5$ such that $G[(V^* \setminus V^{**}) \cup U]$ is either bipartite or $3$-uniform.}
Let~$V^{**}$ be the union of all sets~$V^*_i$ with $i \in \{1,\ldots,5\}$ such that $|V^*_i|=1$.
Then $|V^{**}| \leq 5$.
Consider the graph $H=G[(V^* \setminus V^{**}) \cup U]$.
For every $i \in \{1,\ldots,5\}$ the set~$V_i^*$ either contains zero or at least two vertices in~$H$.
If at most one set~$V_i^*$ contains two vertices in~$H$ then by Claims~\ref{clm3:UWi-indep} and~\ref{clm3:V_i-1WiV_i+1-indep}, $H$ is a bipartite graph and we are done.
It remains to consider the case when two sets~$V_i^*$ and~$V_j^*$ contain two vertices in~$H$.
In this case, by Claim~\ref{clm3:Vistar-large-V_istar-2Vi+2star-small}, $v_i$ and~$v_j$ must be consecutive vertices of the cycle~$C$ and all other sets~$V_k^*$ contain no vertices of~$H$, so we may assume $j=i+1$ and $H=G[V_i^* \cup V_{i+1}^* \cup U]$.
We will show that this graph is $3$-uniform.

By Claims~\ref{clm3:UWi-indep} and~\ref{clm3:V_i-1WiV_i+1-indep}, $V_i^*$, $V_{i+1}^*$ and~$U$ are independent sets.
By Claims~\ref{clm3:T-small} and \ref{clm3:Vistar-large-S-small}, every vertex of $V_i^* \cup V_{i+1}^*$ has exactly one neighbour in~$U$ and every vertex of~$U$ has at most one neighbour in~$V_i^*$ and at most one neighbour in~$V_{i+1}^*$.

If $x \in V_i^*$ and $y \in V_{i+1}^*$ have a common neighbour $z \in U$ then~$x$ must be non-adjacent to~$y$ by Claim~\ref{clm3:V_i-U-nbr-V_i=1Vi+1-distinguishes}.
If $x \in V_i^*$ is non-adjacent to $y \in V_{i+1}^*$ and~$z$ is the neighbour of~$x$ in~$U$ then~$y$ must adjacent to~$z$ by Claim~\ref{clm3:V_i-U-nbr-V_i=1Vi+1-distinguishes}.
Therefore a vertex $x \in V_i^*$ is non-adjacent to a vertex $y \in  V_{i+1}^*$ if and only if they have the same unique neighbour $z \in U$.

Now applying a bipartite complementation between~$V_i^*$ and~$V_{i+1}^*$, we obtain a graph every component of which is an induced subgraph of~$K_3$, with at most one vertex of each component in each of $V_i^*$, $V_{i+1}^*$ and~$U$.
Therefore~$H$ is a $3$-uniform graph.
(In terms of the definition of $3$-uniform graphs, we have $k=3$, $F_k=K_3$ and~$K$ is the $3 \times 3$ matrix which has $K_{1,2}=K_{2,1}=1$ and is zero everywhere else.)
This completes the proof of Claim~\ref{clm3:VstarcupU-done}.

\bigskip
\noindent
Claims~\ref{clm3:VstarU-seperable} and~\ref{clm3:VstarcupU-done} would allow us to remove $G[V^*\cup U]$ from the graph.
As discussed earlier, we do not actually do so at this stage, so that we can preserve the property that~$G$ is prime.
Informally, we may think of the vertices in $V^*\cup U$ as having been ``dealt with'' and we now concern ourselves with the remainder of the graph i.e. the vertices in $C \cup V_1^0 \cup \cdots \cup V_5^0 \cup W_1 \cup \cdots \cup W_5$.
First, we look at the sets~$W_i$ and the edges between these sets.

\newpage
\clm{\label{clm3:Wi-W_i+1-comp-anti}For $i \in \{1,\ldots,5\}$, $W_i$ is either complete or anti-complete to~$W_{i+1}$.}
Suppose, for contradiction, that $x \in W_i$ is adjacent to $y \in W_{i+1}$ and non-adjacent to $y' \in W_{i+1}$.
By Claim~\ref{clm3:UWi-indep}, $y$ is non-adjacent to~$y'$.
Then $G[v_{i-1},v_{i-2},x,y,v_{i+1},y']$ is a $P_2+\nobreak P_4$, a contradiction.
By symmetry, this completes the proof of Claim~\ref{clm3:Wi-W_i+1-comp-anti}.

\medskip
\noindent
In fact, we can strengthen Claim~\ref{clm3:Wi-W_i+1-comp-anti} as follows.

\clm{\label{clm3:Wi-W_i-1W_i+1-comp-anti}For $i \in \{1,\ldots,5\}$, $W_i$ is either complete or anti-complete to $W_{i-1} \cup W_{i+1}$.}
If~$W_{i-1}$ or~$W_{i+1}$ is empty then the claim follows by Claim~\ref{clm3:Wi-W_i+1-comp-anti}.
Now suppose, for contradiction, that $x \in W_i$ has a neighbour $y \in W_{i-1}$ and a non-neighbour $z \in W_{i+1}$.
Then $G[v_{i+2},v_{i+3},v_i,x,y,z]$ or $G[x,y,z,v_{i+1},v_{i+2},v_{i+3}]$ is a $P_2+\nobreak P_4$ if~$y$ is adjacent or non-adjacent to~$z$, respectively.
This contradiction completes the proof of Claim~\ref{clm3:Wi-W_i-1W_i+1-comp-anti}.

\clm{\label{clm3:Wi-W_i+1-P4-free}For $i \in \{1,\ldots,5\}$, $G[W_i \cup W_{i+2}]$ is $P_4$-free.}
If $G[W_i \cup W_{i+2}]$ contains a~$P_4$, say on vertices $w,x,y,z$ then $G[v_{i+3},v_{i+4},w,x,y,z]$ is a $P_2+\nobreak P_4$, a contradiction.
This completes the proof of Claim~\ref{clm3:Wi-W_i+1-P4-free}.

\medskip
\noindent
Since~$W_i$ and~$W_{i+2}$ are independent sets by Claim~\ref{clm3:UWi-indep}, it follows from Claim~\ref{clm3:Wi-W_i+1-P4-free} that every component of $G[W_i\cup W_{i+2}]$ is a complete bipartite graph.
Suppose~$X$ and~$Y$ are independent sets such that every component of $G[X \cup Y]$ is a complete bipartite graph.
We say that a pair $\{X,Y\}$ is {\em non-simple} if the number of non-trivial components (i.e. those containing at least one edge) in $G[X\cup Y]$ is at least two.
Otherwise, we will say that the pair $\{X,Y\}$ is {\em simple}.
Note that if $\{X,Y\}$ is simple then all the edges in $G[X \cup Y]$ can be removed by using at most one bipartite complementation.

\clm{\label{clm3:Wi+2-anti-Wi+3-then-Wi-comp-Wi+2Wi+3}For $i \in \{1,\ldots,5\}$, if~$W_i$, $W_{i+2}$ and~$W_{i-2}$ are all non-empty and~$W_{i+2}$ is anti-complete to~$W_{i-2}$, then~$W_i$ is complete to $W_{i+2} \cup W_{i-2}$.
Furthermore, in this case for all $j \in \{1,\ldots,5\}$, $W_j$ is anti-complete to~$W_{j+1}$ and complete to~$W_{j+2}$.}
Suppose $x \in W_i$, $y \in W_{i+2}$ and $z \in W_{i-2}$ and that $W_{i+2}$ is anti-complete to~$W_{i-2}$.
Suppose, for contradiction, that~$x$ is non-adjacent to~$y$.
Then $G[y,v_{i+2},v_{i-1},v_i,x,z]$ or $G[x,v_i,y,v_{i+2},v_{i-2},z]$ is a $P_2+\nobreak P_4$ if~$x$ is adjacent or non-adjacent to~$z$, respectively.
This contradiction implies that~$x$ is adjacent to~$y$.
We conclude that~$W_i$ is complete to~$W_{i+2}$.
By symmetry, $W_i$ is also complete to~$W_{i-2}$.
Therefore the first part of Claim~\ref{clm3:Wi+2-anti-Wi+3-then-Wi-comp-Wi+2Wi+3} holds.

Now suppose that there is a vertex $w \in W_{i+1}$.
Since~$W_{i-2}$ is anti-complete to~$W_{i+2}$, Claim~\ref{clm3:Wi-W_i-1W_i+1-comp-anti} implies that~$W_{i+2}$ is anti-complete to~$W_{i+1}$ and so~$W_{i+1}$ is anti-complete to~$W_i$.
Applying the first part of Claim~\ref{clm3:Wi+2-anti-Wi+3-then-Wi-comp-Wi+2Wi+3}, we find that~$W_{i+1}$ is complete to~$W_{i-2}$ and if~$W_{i-1}$ is non-empty then~$W_{i+1}$ is complete to~$W_{i-1}$.
Claim~\ref{clm3:Wi+2-anti-Wi+3-then-Wi-comp-Wi+2Wi+3} now follows by symmetry.

\clm{\label{clm3:WiWi+2-and-Wi-2W-non-simple-then-Wi+-1-empty}For $i \in \{1,\ldots,5\}$, if $\{W_i,W_{i+2}\}$ and $\{W_{i-2},W_i\}$ are both non-simple, then~$W_{i-2}$ is complete to~$W_{i+2}$ and $W_{i-1}=W_{i+1}=\emptyset$.}
Suppose $W_i$, $W_{i+2}$ and~$W_{i-2}$ are all non-empty and $\{W_i,W_{i+2}\}$ is non-simple.
By Claim~\ref{clm3:Wi+2-anti-Wi+3-then-Wi-comp-Wi+2Wi+3}, $W_{i-2}$ is not anti-complete to~$W_{i+2}$, so by Claim~\ref{clm3:Wi-W_i+1-comp-anti}, $W_{i-2}$ is complete to~$W_{i+2}$.
Now suppose, for contradiction, that~$W_{i+1}$ contains a vertex~$x$.
By Claim~\ref{clm3:Wi-W_i-1W_i+1-comp-anti}, since~$W_{i-2}$ is complete to~$W_{i+2}$, it follows that~$W_{i+2}$ is complete to~$W_{i+1}$ and therefore~$W_{i+1}$ is complete to~$W_i$.
Since $\{W_i,W_{i+2}\}$ is not simple, there must be adjacent vertices $y \in W_i$ and $z \in W_{i+2}$.
Now $G[x,y,z]$ is a~$K_3$, a contradiction.
Claim~\ref{clm3:WiWi+2-and-Wi-2W-non-simple-then-Wi+-1-empty} follows by symmetry.

\clm{\label{clm3:if-two-non-simple-then-share-a-Wi}$G$ contains at most two non-simple pairs, and if it contains two, they must be as described in Claim~\ref{clm3:WiWi+2-and-Wi-2W-non-simple-then-Wi+-1-empty}.}
Suppose, for contradiction, that the claim is false.
Then for some $i \in \{1,\ldots,5\}$ both $\{W_i,W_{i+2}\}$ and $\{W_{i-1},W_{i+1}\}$ must be non-simple pairs.
Then, by Claim~\ref{clm3:Wi+2-anti-Wi+3-then-Wi-comp-Wi+2Wi+3}, $W_i$ cannot be anti-complete to~$W_{i-1}$, so by Claim~\ref{clm3:Wi-W_i+1-comp-anti}, $W_i$ is complete to~$W_{i-1}$.
Then, by Claim~\ref{clm3:Wi-W_i-1W_i+1-comp-anti}, it follows that~$W_i$ is complete to~$W_{i+1}$.
Since $\{W_{i-1},W_{i+1}\}$ is non-simple, there must be adjacent vertices $x \in W_{i-1}$ and $z \in W_{i+1}$.
Choosing an arbitrary $y \in W_i$, we find that $G[x,y,z]$ is a~$K_3$, a contradiction.
This completes the proof of Claim~\ref{clm3:if-two-non-simple-then-share-a-Wi}.

\bigskip
\noindent
If $\{W_i,W_{i+2}\}$ is a non-simple pair then let~$W_i^+$ be the set of vertices in~$W_i$ that have neighbours in~$W_{i+2}$ and let~$W_{i+2}^-$ be the set of vertices in~$W_{i+2}$ that have neighbours in~$W_i$.
If $\{W_i,W_{i+2}\}$ {\em is} a simple pair then set $W_i^+ = W_{i+2}^-=\emptyset$.
Let $W^*=\bigcup W_i^+ \cup \bigcup W_i^-$.
Next, we look at the edges from the sets~$W_i^+$ and~$W_{i+2}^-$ to the remainder of $G \setminus (V^* \cup U)$ and show how we could apply bipartite complementations to remove these sets from the graph.

\begin{sloppypar}
\clm{\label{clm3:Wi+-Wi--disjoint}For $i \in \{1,\ldots,5\}$, $W_i^+ \cap W_i^- = \emptyset$.}
Indeed, suppose for contradiction that there is a vertex $x \in W_i^+ \cap W_i^-$.
Then $\{W_i,W_{i+2}\}$ and $\{W_{i-2},W_i\}$ must be non-simple pairs and we can choose $y \in W_{i-2}$ and $z \in W_{i+2}$ that are adjacent to~$x$.
By Claim~\ref{clm3:WiWi+2-and-Wi-2W-non-simple-then-Wi+-1-empty}, $W_{i-2}$ is complete to~$W_{i+2}$.
Then $G[x,y,z]$ is a~$K_3$, a contradiction.
Claim~\ref{clm3:Wi+-Wi--disjoint} follows.
\end{sloppypar}

\clm{\label{clm3:WiWi+2-non-simp-Vi-Vi+2}For $i \in \{1,\ldots,5\}$, suppose $\{W_i,W_{i+2}\}$ is a non-simple pair.
Then~$V_{i+2}$ is complete to~$W_{i+2}$ and anti-complete to~$W_i^+$.
Similarly, $V_i$ is complete to~$W_i$ and anti-complete to~$W_{i+2}^-$.}
Suppose, for contradiction, that $x \in V_{i+2}$ has a non-neighbour $z \in W_{i+2}$.
Since $\{W_i,W_{i+2}\}$ is non-simple, there is a vertex $y \in W_i$ that is non-adjacent to~$z$.
Now $G[v_{i+2},z,x,y,v_i,v_{i-1}]$ or $G[v_i,y,z,v_{i+2},v_{i+3},x]$ is a $P_2+\nobreak P_4$ is~$x$ is adjacent or non-adjacent to~$y$, respectively.
This contradiction implies that~$V_{i+2}$ is complete to~$W_{i+2}$.

Now suppose, for contradiction, that $x \in V_{i+2}$ has a neighbour $y \in W_i^+$.
Since $y \in W_i^+$ it must have a neighbour $z \in W_{i+2}$ and by the previous paragraph, $x$ must be adjacent to~$z$.
Now $G[x,y,z]$ is a~$K_3$, a contradiction.
Therefore~$V_{i+2}$ is anti-complete to~$W_i^+$.
Claim~\ref{clm3:WiWi+2-non-simp-Vi-Vi+2} follows by symmetry.

\clm{\label{clm3:WiWi+2-non-simp-Vi+3-Vi+4}For $i \in \{1,\ldots,5\}$, suppose $\{W_i,W_{i+2}\}$ is a non-simple pair.
Then~$V_{i+3}$ is anti-complete to~$W_{i+2}$ and every vertex of~$V_{i+3}$ is either anti-complete to~$W_i$ or complete to~$W_i^+$.
Similarly, $V_{i+4}$ is anti-complete to~$W_i$ and every vertex of~$V_{i+4}$ is either anti-complete to~$W_{i+2}$ or complete to~$W_{i+2}^-$.}
By Claim~\ref{clm3:V_i-1WiV_i+1-indep}, $W_{i+2}$ is anti-complete to~$V_{i+3}$.
Suppose, for contradiction, that the claim does not hold for some vertex $x \in V_{i+3}$.
Then~$x$ must have a neighbour $y \in W_i$ and a non-neighbour $y' \in W_i^+$ and since $\{W_i,W_{i+2}\}$ is non-simple, we may assume that~$y$ and~$y'$ are in different components of $G[W_i \cup W_{i+2}]$.
Indeed, if~$y$ and~$y'$ are in the same component, then there must be a vertex $y'' \in W_i^+$ in a different component of~$G[W_i \cup W_{i+2}]$ to~$y$ and~$y'$; in this case we replace~$y$ or~$y'$ by~$y''$ if~$x$ is adjacent or non-adjacent to~$y''$, respectively.
Since $y' \in W_i^+$ it must have a neighbour $z' \in W_{i+2}$ and then~$z'$ is non-adjacent to~$y$.
Now $G[y',z',y,x,v_{i+4},v_{i+3}]$ is a $P_2+\nobreak P_4$, a contradiction.
Claim~\ref{clm3:WiWi+2-non-simp-Vi+3-Vi+4} follows by symmetry.

\clm{\label{clm3:non-simple-pairs-removable}For $i \in \{1,\ldots,5\}$, every vertex of $V(G) \setminus (V^* \cup U \cup W_i^+ \cup W_{i+2}^-)$ is either complete or anti-complete to~$W_i^+$ and either complete or anti-complete to~$W_{i+2}^-$.}
We may assume that $\{W_i,W_{i+2}\}$ is non-simple, otherwise~$W_i^+$ and~$W_{i+2}^-$ are empty and the claim holds trivially.
By symmetry it is enough to show that every vertex of $V(G) \setminus (V^* \cup U \cup W_i^+ \cup W_{i+2}^-)$ is either complete or anti-complete to~$W_i^+$.
First note that~$v_i$ is complete to~$W_i^+$ and $V(C) \setminus \{v_i\}$ is anti-complete to~$W_i^+$ by definition of~$W_i$.
By Claim~\ref{clm3:V_i-1WiV_i+1-indep}, every vertex of~$V_{i+1}^0$ is anti-complete to~$W_i^+$.
By Claim~\ref{clm3:WiWi+2-non-simp-Vi-Vi+2}, every vertex of~$V_i^0$ is complete to~$W_i^+$ and every vertex of~$V_{i+2}^0$ is anti-complete to~$W_i^+$.
By Claim~\ref{clm3:WiWi+2-non-simp-Vi+3-Vi+4}, every vertex of~$V_{i+3}^0$ and~$V_{i+4}^0$ is either complete or anti-complete to~$W_i^+$.
By Claim~\ref{clm3:UWi-indep}, every vertex of $W_i \setminus W_i^+$ is anti-complete to~$W_i^+$.
By Claim~\ref{clm3:Wi-W_i+1-comp-anti}, every vertex of~$W_{i-1}$ and~$W_{i+1}$ is either complete or anti-complete to~$W_i^+$.
By definition of~$W_i^+$, every vertex of~$W_{i+2} \setminus W_{i+2}^-$ is anti-complete to~$W_i^+$.
Suppose, for contradiction, that a vertex $x \in W_{i-2}$ has a neighbour $y \in W_i^+$.
Since $\{W_i,W_{i+2}\}$ is non-simple, Claim~\ref{clm3:Wi+2-anti-Wi+3-then-Wi-comp-Wi+2Wi+3} implies that~$W_{i-2}$ is not anti-complete to~$W_{i+2}$, so by Claim~\ref{clm3:Wi-W_i+1-comp-anti}, $W_{i-2}$ is complete to~$W_{i+2}$.
Since $y \in W_i^+$ there must be a vertex $z \in V_{i+2}$ that is adjacent to~$y$.
Now $G[x,y,z]$ is a~$K_3$, a contradiction.
Therefore~$W_{i-2}$ is anti-complete to~$W_i^+$.
Claim~\ref{clm3:non-simple-pairs-removable} follows by symmetry.

\clm{\label{clm3:no-large-modules}Every module~$M$ in $G \setminus (U \cup V^* \cup W^*)$ that is an independent set contains at most 512 vertices.}
Recall that~$G$ is prime, so it contains no non-trivial modules.
Let $H= G \setminus (V^* \cup U \cup W^*)$ and suppose, for contradiction, that~$H$ contains a module~$M$ on at least 513 vertices such that~$M$ is an independent set.
By Claim~\ref{clm3:if-two-non-simple-then-share-a-Wi} all but at most two sets of the form~$W_i^+$ are empty and all but at most two sets of the form~$W_i^-$ are empty in~$G$.
By Claim~\ref{clm3:non-simple-pairs-removable}, for all $i \in \{1,\ldots,5\}$, every vertex of~$M$ is either complete or anti-complete to~$W_i^+$ and~$W_i^-$ in~$G$.
By Claim~\ref{clm3:VstarU-seperable}, for all $i \in \{1,\ldots,5\}$, every vertex of~$M$ is either complete or anti-complete to~$V_i^*$ in~$G$ and every vertex of~$M$ is anti-complete to~$U$ in~$G$.
Therefore in~$G$ every vertex of~$M$ has one of at most $2^{(4+5)}=512$ neighbourhoods in $U \cup V^* \cup W^*$.
Choose two vertices $x,y \in M$ that have the same neighbourhood in $U \cup V^* \cup W^*$.
Since~$M$ is a module in~$H$, $x$ and~$y$ are not distinguished by any vertex of $V(H) \setminus M$.
Since~$M$ is an independent set, $x$ and~$y$ are not distinguished by any vertex of $M \setminus \{x,y\}$.
Therefore $\{x,y\}$ is a non-trivial module in~$G$.
This contradiction completes the proof of Claim~\ref{clm3:no-large-modules}.

\clm{\label{clm3:remove-U-Vstar-Wstar}We may assume that $U \cup V^* \cup W^*=\emptyset$.}
By Claim~\ref{clm3:VstarU-seperable}, we may apply at most five bipartite complementations to separate $G[U \cup V^*]$ from the rest of the graph.
Now, by Claim~\ref{clm3:VstarcupU-done}, we can delete at most five vertices from~$G[U \cup V^*]$ to obtain a graph that is either bipartite (in which case it is curious) or $3$-uniform.
Suppose that~$G$ contains a non-simple pair $\{W_i,W_{i+2}\}$ and let $H=G \setminus (V^* \cup U \cup W_i^+ \cup W_{i+2}^-)$.
Then Claim~\ref{clm3:non-simple-pairs-removable} shows that we can remove all edges between $W_i^+ \cup W_{i+2}^-$ and the vertices of~$H$ by applying at most two bipartite complementations.
Furthermore, $G[W_i^+ \cup W_{i+2}^-]$ is a bipartite graph by Claim~\ref{clm3:UWi-indep}, so it is a curious $(K_3,P_2+\nobreak P_4)$-free graph with one part of the $3$-partition empty.
By Claim~\ref{clm3:if-two-non-simple-then-share-a-Wi}, $G$ contains at most two non-simple pairs and if it contains two then by Claim~\ref{clm3:Wi+-Wi--disjoint}, $W_i^+ \cap W_i^- =\emptyset$ for all~$i$.
Hence, Claim~\ref{clm3:non-simple-pairs-removable} implies that by applying at most four bipartite complementations, we can separate~$G[W^*]$ from the rest of $G \setminus (V^* \cup U)$.
We do this by separating the curious graph $G[W_i^+ \cup W_{i+2}^-]$ for each non-simple pair $\{W_i,W_{i+2}\}$ in turn.
We may therefore assume that $U \cup V^* \cup W^*=\emptyset$. 
This completes the proof of Claim~\ref{clm3:remove-U-Vstar-Wstar}.

\bigskip
\noindent
From now on, we assume that $U=\emptyset$ and $\{W_i,W_{i+2}\}$ is simple for all $i \in \{1,\ldots,5\}$.
Note that in the proof of Claim~\ref{clm3:remove-U-Vstar-Wstar} we edit the graph~$G$ in such a way that it may stop being prime.
However, by Claim~\ref{clm3:no-large-modules}, $G$ will still not contain any modules~$M$ that are independent sets on more than 512 vertices, and this property will suffice for the remainder of the proof.

\clm{\label{clm3:2Vi-verts-dominate-Vi+-1}For $i \in \{1,\ldots,5\}$, in every set of 513 vertices in~$V_i$, there are two that together dominate either~$V_{i+1}$ or~$V_{i-1}$.}
Let $X \subseteq V_i$ with $|X|\geq 513$ and note that~$X$ is an independent set by Claim~\ref{clm3:V_i-1WiV_i+1-indep}.
By Claim~\ref{clm3:no-large-modules}, $X$ cannot be a module, so there must be two vertices $x,y \in X$ that are distinguished by a vertex~$z$ outside~$X$.
Without loss of generality assume that~$z$ is adjacent to~$x$, but non-adjacent to~$y$.
By definition of~$V_i$, vertices in the cycle~$C$ cannot distinguish two vertices in the same set~$V_i$, so $z \notin C$.
We will show that $\{x,y\}$ dominates either~$V_{i+1}$ or~$V_{i-1}$.
Suppose, for contradiction, that~$x$ and~$y$ are both non-adjacent to $a \in V_{i-1}$ and both non-adjacent to $b \in V_{i+1}$.
By Claim~\ref{clm3:V_i-1WiV_i+1-indep} vertices in $V_{i-2}\cup V_{i+2} \cup W_{i-1} \cup W_{i+1} \cup (V_i \setminus \{x,y\})$ are anti-complete to $\{x,y\}$.
By symmetry, we may therefore assume that $z \in V_{i+1} \cup W_i \cup W_{i+2}$.
In this case~$z$ is non-adjacent to~$v_{i+1}$, $v_{i-1}$ and~$v_{i-2}$.
Now~$z$ must be adjacent to~$a$, otherwise $G[a,v_{i-2},z,x,v_{i+1},y]$ would be a $P_2+\nobreak P_4$.
Now~$z$ cannot belong to $V_{i+1} \cup W_i$, otherwise $G[a,v_i,z]$ would be a~$K_3$.
Therefore $z \in W_{i+2}$.
Now~$z$ cannot be adjacent to~$b$ otherwise $G[v,v_{i+2},z]$ would be a~$K_3$ and by Claim~\ref{clm3:V_i-1WiV_i+1-indep}, $a$ is non-adjacent to~$b$.
It follows that $G[y,v_{i-1},a,z,v_{i+2},b]$ is a $P_2+\nobreak P_4$.
This contradiction completes the proof of Claim~\ref{clm3:2Vi-verts-dominate-Vi+-1}.

\clm{\label{clm3:Vi-dominates-Vi-1orVi+1}By deleting at most 512 vertices from each set~$V_i$ we may assume that each vertex of~$V_i$ dominates either~$V_{i-1}$ or~$V_{i+1}$.}
Note that~$V_{i-1}$ is anti-complete to~$V_{i+1}$ by Claim~\ref{clm3:V_i-1WiV_i+1-indep}.
For each $x \in V_{i-1}$ and $y\in V_{i+1}$, let~$A^{x,y}$ denote the set of vertices in~$V_i$ that are non-adjacent to both~$x$ and~$y$.
Consider an $x \in V_{i-1}$ and a~$y\in V_{i+1}$ such that~$A^{x,y}$ is non-empty, say $a \in A^{x,y}$.
(Note that if~$A^{x,y}$ is empty for all~$x$ and~$y$ then we are done.)
If a vertex $b \in V_i$ is adjacent to exactly one of~$x$ and~$y$, say~$b$ is adjacent to~$x$, then $G[v_{i+2},y,a,v_{i-1},b,x]$ is a $P_2+\nobreak P_4$, a contradiction.
This contradiction shows that if~$A^{x,y}$ is non-empty then $N(x) \cap V_i=N(y)\cap V_i=V_i \setminus A^{x,y}$.

Now choose~$x$ and~$y$ such that~$|A^{x,y}|$ is maximum.
Suppose, for contradiction, that there is a vertex $b \in V_i \setminus A^{x,y}$ with a non-neighbour $x' \in V_{i-1}$ and a non-neighbour $y' \in V_{i+1}$.
Note that~$b$ must be adjacent to both~$x$ and~$y$ and $b \in A^{x',y'}$.
Furthermore, since~$|A^{x,y}|$ is maximum, there must be a vertex $a \in A^{x,y} \setminus A^{x',y'}$, so~$a$ is adjacent to both~$x'$ and~$y'$.
Now $G[a,x',x,b,y,v_{i+2}]$ is a $P_2+\nobreak P_4$, a contradiction.
It follows that every vertex in $V_i \setminus A^{x,y}$ dominates either~$V_{i-1}$ or~$V_{i+1}$.
By Claim~\ref{clm3:2Vi-verts-dominate-Vi+-1}, $|A^{x,y}| \leq 512$, so we delete the vertices of~$A^{x,y}$.
This completes the proof of Claim~\ref{clm3:Vi-dominates-Vi-1orVi+1}.

\bigskip
\noindent
By Claim~\ref{clm3:Vi-dominates-Vi-1orVi+1}, we may assume that every vertex of~$V_i$ dominates either~$V_{i-1}$ or~$V_{i+1}$.
Therefore, we can partition each set~$V_i$ into three (possibly empty) subsets as follows:

\begin{itemize}
\item $V^0_i$ the set of vertices in~$V_i$ that dominate both~$V_{i-1}$ and~$V_{i+1}$,
\item $V^-_i$ the set of vertices in $V_i\setminus V^0_i$ that dominate~$V_{i-1}$ (and so have non-neighbours in~$V_{i+1}$),
\item $V^+_i$ the set of vertices in $V_i\setminus V^0_i$ that dominate~$V_{i+1}$ (and so have non-neighbours in~$V_{i-1}$).
\end{itemize}
By definition of this partition, if $x\in V_i$ is non-adjacent to $y\in V_{i+1}$, then $x \in V^-_i$ and $y \in V^+_{i+1}$.
Moreover, {\em every} vertex of~$V^-_i$ has a non-neighbour in~$V^+_{i+1}$ and vice versa. 
Thus, the vertices of $V_1\cup\cdots\cup V_5$ are partitioned into 15 subsets (some or all of which may be empty).
Note that for $i \in \{1,\ldots,5\}$ the sets $V_i^0$, $V^-_i$, $V^+_i$, $V_{i+2}^0$, $V^-_{i+2}$ and~$V^+_{i+2}$ are pairwise anti-complete by Claim~\ref{clm3:V_i-1WiV_i+1-indep}.
Furthermore, for $i \in \{1,\ldots,5\}$, $V_i^0$ and~$V^+_i$ are complete to $V_{i+1}^0$, $V^-_{i+1}$ and~$V^+_{i+1}$, and $V^-_{i+1}$ is complete to $V_i^0$ and~$V^-_i$ by definition.
Therefore nearly every pair of these subsets is either complete or anti-complete to each other.
The only possible exceptions are the five disjoint pairs of the form $\{V^-_i,V^+_{i+1}\}$.

\medskip
Recall that for two independent sets of vertices~$X$ and~$Y$ such that $G[X \cup Y]$ is $P_4$-free, the components of $G[X \cup Y]$ are complete bipartite graphs and if $G[X \cup Y]$ contains at most one non-trivial component, we say that the pair $\{X,Y\}$ is simple and otherwise it is non-simple.
Recall that if $\{X,Y\}$ is simple then we can remove all edges between~$X$ and~$Y$ by applying at most one bipartite complementation.
By Claim~\ref{clm3:remove-U-Vstar-Wstar} every pair $\{W_i,W_j\}$ is simple and for the fifteen sets considered in the previous paragraph, the only pairs of them that can be non-simple are the ones of the form $\{V^-_i,V^+_{i+1}\}$.

Next, we will analyse the edges between sets of the form~$W_i$ and sets of the form~$V_j$ (and their subsets $V^0_j$, $V^-_j$ and~$V^+_j$) for $i,j \in \{1,\ldots,5\}$.
We will then use bipartite complementations to partition the graph into induced subgraphs that are curious.
First, recall that~$W_i$ is anti-complete to~$V_{i-1}$ and~$V_{i+1}$ by Claim~\ref{clm3:V_i-1WiV_i+1-indep}.
Also, by Claim~\ref{clm3:x-in-Vi-triv-to-Wi}, $\{W_i,V_i\}$ is simple.
Therefore, the only sets~$V_j$ to which~$W_i$ can have a non-simple connection are~$V_{i+2}$ and~$V_{i+3}$.

\begin{sloppypar}
\clm{\label{clm3:W_iCi+2Vi+3-no-3P1}For $i \in \{1,\ldots,5\}$, if $x\in W_i$, $y\in V_{i+2}$ and $z\in V_{i+3}$ then $G[x,y,z]$ is not a~$3P_1$.}
Indeed, if $x\in W_i$, $y\in V_{i+2}$ and $z\in V_{i+3}$ and $G[x,y,z]$ is a~$3P_1$ then $G[x,v_i,y,v_{i+3},v_{i+2},z]$ is a $P_2+\nobreak P_4$, a contradiction.
This completes the proof of Claim~\ref{clm3:W_iCi+2Vi+3-no-3P1}.
\end{sloppypar}

\clmnonewline{\label{clm3:x-in-Wi-nbr-in-W0+_i+2-then-anti-Vi+3-comp-V-_i+2}For $i \in \{1,\ldots,5\}$:
\begin{enumerate}[(i)]
\item If $x\in W_i$ has a neighbour in $V^0_{i+2}\cup V^+_{i+2}$, then~$x$ is anti-complete to~$V_{i+3}$ and complete to~$V^-_{i+2}$.
\item If $x\in W_i$ has a neighbour in $V^0_{i+3}\cup V^-_{i+3}$, then~$x$ is anti-complete to~$V_{i+2}$ and complete to~$V^+_{i+3}$.
\end{enumerate}
}
\noindent
Suppose $x\in W_i$ has a neighbour $y \in V^0_{i+2}\cup V^+_{i+2}$.
If $z \in V_{i+3}$ then~$y$ is adjacent to~$z$, so~$x$ must be non-adjacent to~$z$, otherwise $G[x,y,z]$ would be a~$K_3$.
Therefore~$x$ is anti-complete to~$V_{i+3}$.
If $y' \in V^-_{i+2}$ then~$y'$ has a non-neighbour $z' \in V_{i+3}$.
Note that~$z'$ is non-adjacent to~$x$.
Now~$x$ must be adjacent to~$y'$, otherwise $G[x,y',z']$ would be a~$3P_1$, contradicting Claim~\ref{clm3:W_iCi+2Vi+3-no-3P1}.
It follows that~$x$ is complete to~$V^-_{i+2}$.
Claim~\ref{clm3:x-in-Wi-nbr-in-W0+_i+2-then-anti-Vi+3-comp-V-_i+2} follows by symmetry.

\bigskip
By Claim~\ref{clm3:x-in-Wi-nbr-in-W0+_i+2-then-anti-Vi+3-comp-V-_i+2} we can partition~$W_i$ into three (possibly empty) subsets as follows:
\begin{itemize}
\item $W^2_i$ the set of vertices in~$W_i$ that have neighbours in $V^0_{i+2}\cup V^+_{i+2}$ (and are therefore anti-complete to~$V_{i+3}$ and complete to~$V^-_{i+2}$),
\item $W^3_i$ the set of vertices in~$W_i$ that have neighbours in $V^0_{i+3}\cup V^-_{i+3}$ (and are therefore anti-complete to~$V_{i+2}$ and complete to~$V^+_{i+3}$) and
\item $W^0_i$ the set of vertices in $W_i\setminus (W^2_i\cup W^3_i)$ (which are therefore anti-complete to $V^0_{i+2}\cup V^+_{i+2}\cup V^0_{i+3}\cup V^-_{i+3}$).
\end{itemize}

Recall that the only sets~$V_j$ to which~$W_i$ can have a non-simple connection are~$V_{i+2}$ and~$V_{i+3}$.
Therefore, by Claim~\ref{clm3:x-in-Wi-nbr-in-W0+_i+2-then-anti-Vi+3-comp-V-_i+2}, for $i,j \in \{1,\ldots,5\}$ if a set in $\{W_i^2, W_i^3, W_i^0\}$ is non-simple to a set in $\{V_j^-, V_j^+, V_j^0\}$ then it must be one of the following pairs:
\begin{itemize}
\item $\{W_i^2,V^0_{i+2}\}$,
\item $\{W_i^2,V^+_{i+2}\}$,
\item $\{W_i^3,V^0_{i+3}\}$,
\item $\{W_i^3,V^-_{i+3}\}$,
\item $\{W_i^0,V^-_{i+2}\}$ or 
\item $\{W_i^0,V^+_{i+3}\}$.
\end{itemize}
Recall among pairs of sets of the form $V_j^-$, $V_j^+$, $V_j^0$, $W_j^2$, $W_j^3$, $W_j^0$ the only possible other non-simple pairs are the ones of the form $\{V^-_i,V^+_{i+1}\}$ (see also \figurename~\ref{fig:non-simple}).

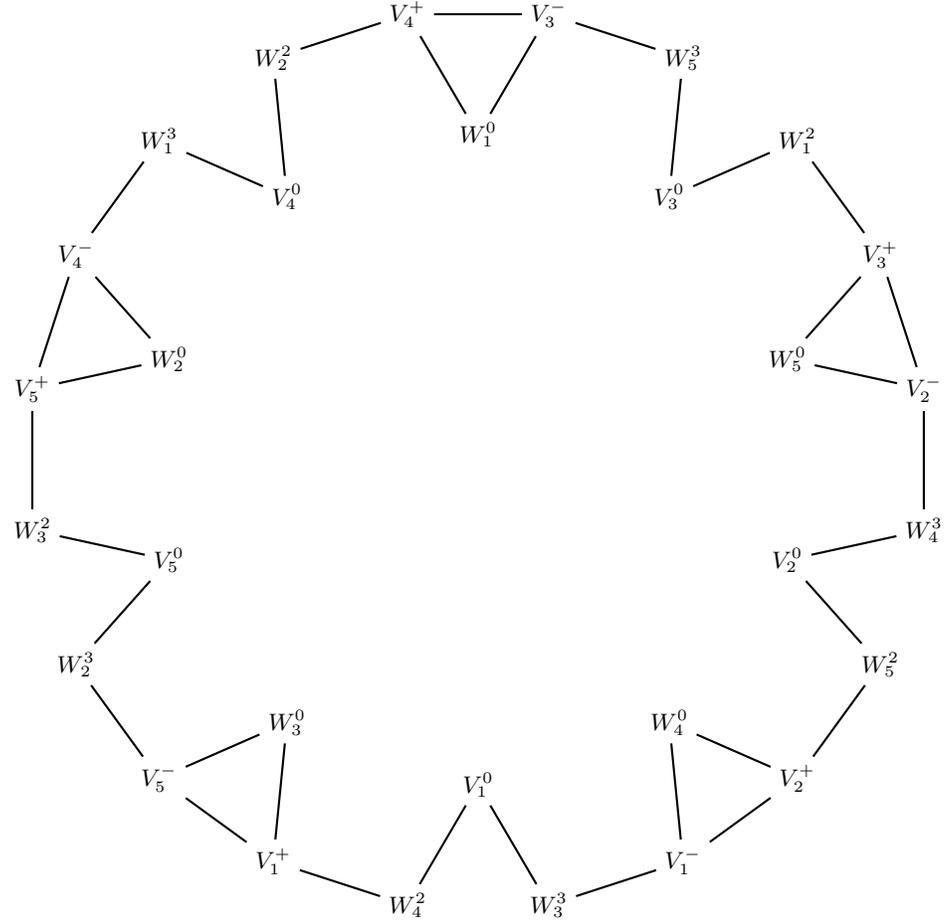
\begin{figure}
\begin{center}
\begin{tikzpicture}[rotate=81]
\draw node (x0) at (0*18:6) {$V_3^-$};
\draw node (x1) at (1*18:6) {$V_4^+$};
\draw node (x2) at (2*18:6) {$W_2^2$};
\draw node (x3) at (3*18:6) {$W_1^3$};
\draw node (x4) at (4*18:6) {$V_4^-$};
\draw node (x5) at (5*18:6) {$V_5^+$};
\draw node (x6) at (6*18:6) {$W_3^2$};
\draw node (x7) at (7*18:6) {$W_2^3$};
\draw node (x8) at (8*18:6) {$V_5^-$};
\draw node (x9) at (9*18:6) {$V_1^+$};
\draw node (x10) at (10*18:6) {$W_4^2$};
\draw node (x11) at (11*18:6) {$W_3^3$};
\draw node (x12) at (12*18:6) {$V_1^-$};
\draw node (x13) at (13*18:6) {$V_2^+$};
\draw node (x14) at (14*18:6) {$W_5^2$};
\draw node (x15) at (15*18:6) {$W_4^3$};
\draw node (x16) at (16*18:6) {$V_2^-$};
\draw node (x17) at (17*18:6) {$V_3^+$};
\draw node (x18) at (18*18:6) {$W_1^2$};
\draw node (x19) at (19*18:6) {$W_5^3$};
\draw [thick] (x19) -- (x0) -- (x1) -- (x2)
              (x3) -- (x4) -- (x5) -- (x6)
              (x7) -- (x8) -- (x9) -- (x10)
              (x11) -- (x12) -- (x13) -- (x14)
              (x15) -- (x16) -- (x17) -- (x18);

\draw node (y0) at  (0*36+9:5-0.67738110376) {$W_1^0$}; 
\draw node (y1) at  (1*36+9:5-0.67738110376) {$V_4^0$};
\draw node (y2) at  (2*36+9:5-0.67738110376) {$W_2^0$};
\draw node (y3) at  (3*36+9:5-0.67738110376) {$V_5^0$};
\draw node (y4) at  (4*36+9:5-0.67738110376) {$W_3^0$};
\draw node (y5) at  (5*36+9:5-0.67738110376) {$V_1^0$};
\draw node (y6) at  (6*36+9:5-0.67738110376) {$W_4^0$};
\draw node (y7) at  (7*36+9:5-0.67738110376) {$V_2^0$};
\draw node (y8) at  (8*36+9:5-0.67738110376) {$W_5^0$};
\draw node (y9) at  (9*36+9:5-0.67738110376) {$V_3^0$};

\draw [thick]
(x0) -- (y0) -- (x1)
(x2) -- (y1) -- (x3)
(x4) -- (y2) -- (x5)
(x6) -- (y3) -- (x7)
(x8) -- (y4) -- (x9)
(x10) -- (y5) -- (x11)
(x12) -- (y6) -- (x13)
(x14) -- (y7) -- (x15)
(x16) -- (y8) -- (x17)
(x18) -- (y9) -- (x19)

;
\end{tikzpicture}
\end{center}
\caption{\label{fig:non-simple}The connections between pairs of sets of the form $V_j^-$, $V_j^+$, $V_j^0$, $W_j^2$, $W_j^3$ and $W_j^0$.
An edge is shown between two such sets if it is possible for this pair of sets to be non-simple.
If an edge is not shown between two sets, then these sets form a simple pair and therefore all edges between them can be removed by using at most one bipartite complementation.}
\end{figure}

Notice that~$W^2_i$ and~$W^3_i$ are disjoint by Claim~\ref{clm3:x-in-Wi-nbr-in-W0+_i+2-then-anti-Vi+3-comp-V-_i+2}.
This partition allows us to distinguish two kinds of induced subgraphs of~$G$.
\begin{enumerate}
\item If for some $i \in \{1,\ldots,5\}$ $W^3_i$ and~$W^2_{i+1}$ are both non-empty then we form a subgraph of the {\em first} kind $H^1_i = G[W^3_i\cup W^2_{i+1}\cup V^0_{i+3}]$. (If $W^3_i$ or~$W^2_{i+1}$ is empty, then we let~$H^1_i$ be the empty graph~$(\emptyset,\emptyset)$.)

\item For $i \in \{1,\ldots,5\}$ we let~$H^2_i$ be the graph induced on~$G$ by $W^0_i \cup V^-_{i+2} \cup V^+_{i+3}$.
If~$H_{i-1}^1$ is empty, but~$W^3_{i-1}$ is non-empty then $W^2_i=\emptyset$.
In this case we also add $W^3_{i-1} \cup V^0_{i+2}$ to~$H^2_i$.
Similarly, if~$H_i^1$ is empty, but~$W^2_{i+1}$ is non-empty then $W^3_i=\emptyset$.
In this case we also add $W^2_{i+1} \cup V^0_{i+3}$ to~$H^2_i$.
We say that~$H^2_i$ is a graph of the {\em second} kind.
\end{enumerate}

In the next two claims, we show how to disconnect both kinds of graph~$H_i^j$ from the rest of the graph.

\clm{\label{clm3:H1_i-seperable}For $i \in \{1,\ldots,5\}$ the graph~$H^1_i$ can be separated from the rest of the graph using finitely many bipartite complementations.}
We will show that for every set $V_j^-$, $V_j^+$, $V_j^0$, $W_j^2$, $W_j^3$, $W_j^0$ that is not in $\{W^3_i,\allowbreak W^2_{i+1},V^0_{i+3}\}$, this set is simple to each of the three sets in $\{W^3_i,W^2_{i+1},V^0_{i+3}\}$, in which case we can separate~$H^1_i$ by applying at most one bipartite complementation between every pair of such sets.
Taking into account Claim~\ref{clm3:x-in-Wi-nbr-in-W0+_i+2-then-anti-Vi+3-comp-V-_i+2}, recall that~$W^3_i$ is simple to every set except~$V^0_{i+3}$ and~$V^-_{i+3}$, $W^2_{i+1}$ is simple to every set except~$V^0_{i+3}$ and~$V^+_{i+3}$ and~$V^0_{i+3}$ is simple to every set except~$W^3_i$ and~$W^2_{i+1}$.
Thus, it suffices to show that~$W^3_i$ is simple to~$V^-_{i+3}$ and~$W^2_{i+1}$ is simple to~$V^+_{i+3}$.
By symmetry, we need only show that~$W^2_{i+1}$ is simple to~$V^+_{i+3}$.
If~$W^2_{i+1}$ is anti-complete to~$V^+_{i+3}$ then we are done.
Otherwise, choose a vertex $x\in W^2_{i+1}$ that has a neighbour $y\in V^+_{i+3}$.
Note that~$W^3_i$ is non-empty, by assumption (otherwise~$H^1_i$ is empty and we are done).
By Claim~\ref{clm3:x-in-Wi-nbr-in-W0+_i+2-then-anti-Vi+3-comp-V-_i+2}, $V^+_{i+3}$ is complete to~$W^3_i$.
Therefore~$x$ must be anti-complete to~$W_i^3$, otherwise for $z \in W_i^3$, $G[x,y,z]$ would be a~$K_3$, a contradiction.
We will show that~$x$ is complete to~$V^+_{i+3}$.
Suppose, for contradiction, that~$x$ has a non-neighbour $y' \in V^+_{i+3}$.
By definition of~$V^+_{i+3}$ and~$V^-_{i+2}$, $y'$ must have a non-neighbour $z' \in V^-_{i+2}$.
Let $u \in W_i^3$ and note that~$u$ is non-adjacent to~$z'$ and adjacent to~$y'$ by Claim~\ref{clm3:x-in-Wi-nbr-in-W0+_i+2-then-anti-Vi+3-comp-V-_i+2}.
Then $G[u,y',x,v_{i+1},z',v_{i+3}]$ is a $P_2+\nobreak P_4$.
This contradiction shows that if $x \in W^2_{i+1}$ has a neighbour in~$V^+_{i+3}$ then~$x$ is complete to~$V^+_{i+3}$.
Therefore, $\{W^2_{i+1},V^+_{i+3}\}$ is simple.
Similarly, $\{W^3_i,V^-_{i+3}\}$ is simple.
Applying bipartite complementations between $W^3_i$\&$\{v_i\}$, $W^2_{i+1}$\&$\{v_{i+1}\}$ and $V^0_{i+3}$\&$\{v_{i+2},v_{i+4}\}$ removes all edges between the vertices of~$H^1_i$ and the vertices of the cycle~$C$.
Every set $V_j^-$, $V_j^+$, $V_j^0$, $W_j^2$, $W_j^3$, $W_j^0$ for $j \in \{1,\ldots,5\}$ that is not in $\{W^3_i,W^2_{i+1},V^0_{i+3}\}$ is simple to every set in $\{W^3_i,W^2_{i+1},V^0_{i+3}\}$, so applying at most one bipartite complementation between each such pair of sets, we can remove all remaining edges from~$H_i^1$ to the rest of the graph.
This completes the proof of Claim~\ref{clm3:H1_i-seperable}.

\clm{\label{clm3:H2_i-seperable}For $i \in \{1,\ldots,5\}$ the graph~$H^2_i$ can be separated from the rest of the graph using finitely many bipartite complementations.}
The graph~$H^2_i$ contains $W^0_i \cup V^-_{i+2} \cup V^+_{i+3}$.
Taking into account Claim~\ref{clm3:x-in-Wi-nbr-in-W0+_i+2-then-anti-Vi+3-comp-V-_i+2}, recall that the set~$W^0_i$ can only be non-simple to~$V^-_{i+2}$ or~$V^+_{i+3}$, the set~$V^-_{i+2}$ can only be non-simple to~$V^+_{i+3}$, $W_i^0$ and~$W_{i-1}^3$, and the set~$V^+_{i+3}$ can only be non-simple to~$V^-_{i+2}$, $W_i^0$ and~$W_{i+1}^2$.
Thus we only need to concern ourselves with the case where the pair $\{V^-_{i+2},W_{i-1}^3\}$ or the pair $\{V^+_{i+3},W_{i+1}^2\}$ is not simple.
By symmetry, it is enough to consider the second of these cases.
If~$W_{i+1}^2$ is in the graph~$H_i^1$ then we are done by Claim~\ref{clm3:H1_i-seperable}.
Therefore we may assume that~$W_{i+1}^2$ is non-empty, but not contained in the graph~$H_i^1$, so $W^3_i=\emptyset$.
Thus we add $W^2_{i+1} \cup V^0_{i+3}$ to~$H^2_i$.
Now~$W^2_{i+1}$ can only be non-simple to~$V^0_{i+3}$ and~$V^+_{i+3}$, both of which are in~$H^2_i$ and $V^0_{i+3}$ can only be non-simple to~$W^2_{i+1}$, which is in~$H^2_i$ and~$W^3_i$, which is empty.
By symmetry, it follows that we can use a bounded number of bipartite complementations to disconnect~$H^2_i$ from all sets of the form $V_j^-$, $V_j^+$, $V_j^0$, $W_j^2$, $W_j^3$, $W_j^0$ outside~$H_i^2$.
Finally, applying a bipartite complementation between each of the sets in~$H^2_i$ and their neighbourhood in the cycle~$C$ disconnects~$H^2_i$ from the remainder of the graph.
This completes the proof of Claim~\ref{clm3:H2_i-seperable}.

\medskip
In the next two claims, we show graphs of the first kind are curious and graphs of the second kind can be partitioned into two curious induced subgraphs by applying at most two bipartite complementations.

\clm{\label{clm3:H1_i-curious}For $i \in \{1,\ldots,5\}$, the graph~$H^1_i$ is bipartite and therefore curious.}
The sets~$W^3_i$, $W^2_{i+1}$ and~$V^0_{i+3}$ are independent by Claim~\ref{clm3:V_i-1WiV_i+1-indep}.
By Claim~\ref{clm3:Wi-W_i+1-comp-anti}, $W^3_i$ is complete or anti-complete to~$W^2_{i+1}$.
If~$W^3_i$ is anti-complete to~$W^2_{i+1}$, then $W^3_i\cup W^2_{i+1}$ and~$V^0_{i+3}$ are independent sets, so~$H^1_i$ is bipartite.
If~$W^3_i$ is complete to~$W^2_{i+1}$, then every vertex of~$V^0_{i+3}$ has a neighbour in at most one of~$W^3_i$ and~$W^2_{i+1}$ (since~$G$ is $K_3$-free), and hence~$H^1_i$ is again bipartite.
Since bipartite graphs are curious graphs where one of the partition classes is empty, this completes the proof of Claim~\ref{clm3:H1_i-curious}.

\clm{\label{clm3:H2_i-curious}For $i \in \{1,\ldots,5\}$, we can apply at most two bipartite complementation operations to~$H^2_i$ to obtain the disjoint union of at most two curious induced subgraphs of~$H^2_i$.}
The sets~$W^0_i$, $V^-_{i+2}$ and~$V^+_{i+3}$ are independent by Claim~\ref{clm3:V_i-1WiV_i+1-indep}.
If~$H^2_i$ only contains the vertices in $W^0_i \cup V^-_{i+2} \cup V^+_{i+3}$ then it is curious by Claim~\ref{clm3:W_iCi+2Vi+3-no-3P1}.
We may therefore assume that $W^3_{i-1}$ and/or $W^2_{i+1}$ belong to~$H^2_i$.
Note that $X=(W^3_{i-1} \cup V^0_{i+3} \cup V^+_{i+3})\cap V(H^2_i)$ and $Y=(W^2_{i+1} \cup V^0_{i+2} \cup V^-_{i+2})\cap V(H^2_i)$ are independent sets by Claim~\ref{clm3:V_i-1WiV_i+1-indep}.
Therefore if~$W^0_i=\emptyset$ then~$H^2_i$ is bipartite, so it is a curious graph, where one part of the $3$-partition is empty.
If~$W^0_i$ is not empty, then it is either complete or anti-complete to $W^3_{i-1}\cup W^2_{i+1}$ by Claim~\ref{clm3:Wi-W_i-1W_i+1-comp-anti}.

If~$W^0_i$ is complete to $W^3_{i-1}\cup W^2_{i+1}$ then taking~$X$ and~$Y$ as above, and setting $Z=W^0_i$, we find that~$H^2_i$ is a curious graph, with partition~$X,Y,Z$.
Indeed, suppose for contradiction that there are vertices $x \in X$, $y \in Y$ and $z \in Z$ with $G[x,y,z]$ forming a~$3P_1$.
Since $z\in W^0_i$, it follows that $x \notin W^3_{i-1}$ and $y \notin W^2_{i+1}$.
Therefore we have $x \in V_{i+2}$, $y \in V_{i+3}$ and $z \in W_i$ such that~$G[x,y,z]$ is a~$3P_1$, which contradicts Claim~\ref{clm3:W_iCi+2Vi+3-no-3P1}.
This contradiction shows that if~$W^0_i$ is complete to $W^3_{i-1}\cup W^2_{i+1}$ then~$H^2_i$ is indeed a curious graph.

Finally, suppose that~$W^0_i$ is anti-complete to $W^3_{i-1}\cup W^2_{i+1}$.
By Claim~\ref{clm3:remove-U-Vstar-Wstar}, we may assume that~$W^3_{i-1}$ is simple to~$W^2_{i+1}$, so $G[W^3_{i-1}\cup W^2_{i+1}]$ contains at most one non-trivial component, and if such a component exists, it must be a complete bipartite graph by Claim~\ref{clm3:Wi-W_i+1-P4-free}.
Let~$W^{3*}_{i-1}$ (resp.~$W^{2*}_{i+1}$) be the set of vertices in~$W^3_{i-1}$ (resp.~$W^2_{i+1}$) that have neighbours in~$W^2_{i+1}$ (resp.~$W^3_{i-1}$).
Since $G[W^{3*}_{i-1} \cup W^{2*}_{i+1}]$ is a complete bipartite graph, it follows that~$W^{3*}_{i-1}$ is complete to~$W^{2*}_{i+1}$.

We will show that if~$W^{3*}_{i-1}$ and~$W^{2*}_{i+1}$ are both in~$H_i^2$, then we can disconnect $H_i^2[W^{3*}_{i-1} \cup W^{2*}_{i+1}]$ from the rest of~$H_i^2$ by applying at most two bipartite complementations.
Indeed, suppose $x \in W^{3*}_{i-1}$ has a neighbour $y \in W^{2*}_{i+1}$ and let~$u \in W_i^0$.
By definition of~$W^3_{i-1}$ there is a vertex $z \in V_{i+2}^0 \cup V_{i+2}^-$ that is adjacent to~$x$.
Then~$y$ must be non-adjacent to~$z$ by Claim~\ref{clm3:V_i-1WiV_i+1-indep}.
If~$z$ is non-adjacent to~$u$ then $G[v_i,u,y,x,z,v_{i-2}]$ is a $P_2+\nobreak P_4$, a contradiction.
Therefore~$z$ is adjacent to~$u$.
By definition of~$W_i^0$, it follows that $z \notin V_{i+2}^0$ and so $z \in V_{i+2}^-$.
We conclude that~$W^{3*}_{i-1}$ is anti-complete to~$V_{i+2}^0$, and by symmetry, $W^{2*}_{i+1}$ is anti-complete to~$V_{i+3}^0$.
Now suppose that there is a vertex $z' \in V_{i+2}^-$ such that~$x$ is non-adjacent to~$z'$.
If~$z'$ is adjacent to~$u$ then $G[x,v_{i-1},u,z',v_{i+1},v_{i+2}]$ is a $P_2+\nobreak P_4$.
Therefore~$z'$ must be non-adjacent to~$u$.
We conclude that every vertex in~$V_{i+2}^-$ is adjacent to~$x$ if and only if it is adjacent to~$u$.
Since~$x$ was chosen from~$W^{3*}_{i-1}$ arbitrarily, this means that every vertex in~$V_{i+2}^-$ is either complete or anti-complete to~$W^{3*}_{i-1}$.
Therefore, applying at most one bipartite complementation, we can remove all edges from~$V_{i+2}^-$ to~$W^{3*}_{i-1}$
Recall that~$V_{i+2}^0$ is anti-complete to~$W^{3*}_{i-1}$.
By assumption, $W^{3*}_{i-1}$ is anti-complete to~$W_i^0$, complete to~$W^{2*}_{i+1}$ and anti-complete to $W^2_{i+1}\setminus W^{2*}_{i+1}$.
By Claim~\ref{clm3:V_i-1WiV_i+1-indep}, $W^{3*}_{i-1}$ is anti-complete to~$V_{i+3}$.
Therefore, applying at most one bipartite complementation, we can remove all edges between vertices in~$W^{3*}_{i-1}$ and vertices in $V(H^2_i) \setminus (W^{3*}_{i-1} \cup W^{2*}_{i+1})$.
By symmetry, applying at most one bipartite complementation, we can remove all edges between vertices in~$W^{2*}_{i+1}$ and vertices in $V(H^2_i) \setminus (W^{3*}_{i-1} \cup W^{2*}_{i+1})$.
This disconnects $G[W^{3*}_{i-1} \cup W^{2*}_{i+1}]$ from the rest of~$H^2_i$.
Since $G[W^{3*}_{i-1} \cup W^{2*}_{i+1}]$ is a bipartite graph, it is curious.
We may therefore assume that~$W^3_{i-1}$ is anti-complete to~$W^2_{i+1}$.

We set $X'=(V^0_{i+3} \cup V^+_{i+3})\cap V(H^2_i)$, $Y'=(V^0_{i+2} \cup V^-_{i+2})\cap V(H^2_i)$ and $Z'=(W_i^0 \cup W^3_{i-1}\cup W^2_{i+1})\cap V(H^2_i)$ and note that~$X'$, $Y'$ and~$Z'$ are independent sets.
Now suppose, for contradiction, that~$H^2_i$ is not curious with respect to the partition $(X',Y',Z')$.
Then there must be vertices $x\in X'$, $y\in Y'$ and $z\in Z'$ such that $G[x,y,z]$ is a~$3P_1$.
Since~$x$ and~$y$ are non-adjacent, it follows that $x \in V^+_{i+3}$ and $y \in V^-_{i+2}$.
By Claim~\ref{clm3:W_iCi+2Vi+3-no-3P1}, $z \notin W^0_i$.
By symmetry, we may assume $z \in W^3_{i-1}$.
Let~$w$ be an arbitrary vertex of~$W_i^0$.
Then~$w$ must be non-adjacent to~$y$, otherwise $G[v_{i-1},z,w,y,v_{i+1},v_{i+2}]$ is a $P_2+\nobreak P_4$.
By Claim~\ref{clm3:W_iCi+2Vi+3-no-3P1}, $w$ must therefore be adjacent to~$x$.
Then $G[v_{i+1},y,w,x,v_{i-1},z]$ is a $P_2+\nobreak P_4$, a contradiction.
It follows that~$H^2_i$ is indeed a curious graph with respect to the partition $(X',Y',Z')$.
This completes the proof of Claim~\ref{clm3:H2_i-curious}.

\medskip
\noindent
Applying Claims~\ref{clm3:H1_i-seperable}, \ref{clm3:H2_i-seperable}, \ref{clm3:H1_i-curious} and~\ref{clm3:H2_i-curious}, we remove all vertices of~$V_i$ and~$W_i$ except maybe for some sets~$V_i^0$.
By definition, each set~$V_i^0$ is complete to $V_{i-1}^0 \cup V_{i+1}^0$ and by Claim~\ref{clm3:V_i-1WiV_i+1-indep} it is anti-complete to~$V_{i-2}^0 \cup V_{i+2}^0$.
We delete the five vertices of the original cycle~$C$ and then apply a bipartite complementation between any pair of non-empty sets~$V_i^0$ and~$V_{i+1}^0$.
This will yield an independent set, which is a curious $(K_3,P_2+\nobreak P_4)$-free graph where two parts of the partition are empty.
It remains only to count the number of operations applied and the number of obtained curious and $3$-uniform graphs.

\bigskip
\noindent
As explained in the proof of Claim~\ref{clm3:remove-U-Vstar-Wstar}, by Claim~\ref{clm3:VstarU-seperable}, we apply at most five bipartite complementations to separate $G[U \cup V^*]$ from the rest of the graph.
We then delete at most five vertices from $G[U \cup V^*]$ to obtain a graph that is either curious or $3$-uniform.
Next, as also explained in the proof of Claim~\ref{clm3:remove-U-Vstar-Wstar}, we can separate~$G[W^*]$ from the rest of $G \setminus (V^* \cup U)$ by applying $2\times 2=4$ bipartite complementations and obtain at most two curious graphs.
Next, in Claim~\ref{clm3:Vi-dominates-Vi-1orVi+1}, for each $i \in \{1,\ldots,5\}$, we delete at most 512 vertices.
Deleting the five vertices of the original cycle~$C$ yields five more vertex deletions.
Recall that for $i \in \{1,\ldots,5\}$, we partitioned~$W_i$ into $W_i^0$, $W_i^2$ and~$W_i^3$ and partitioned~$V_i$ into $V_i^0$, $V_i^-$ and~$V_i^+$, yielding 30 subsets of vertices altogether.
In Claims~\ref{clm3:H1_i-seperable} and~\ref{clm3:H2_i-seperable}, we apply bipartite complementations between certain pairs of these subsets to separate the graphs of both kinds.
Thus, in these claims we apply at most $\binom{30}{2}=435$ bipartite complementations.
For each $i \in \{1,\ldots,5\}$, Claim~\ref{clm3:H1_i-curious} says that~$H_i^1$ is a curious graph.
In Claim~\ref{clm3:H2_i-curious}, for each $i \in \{1,\ldots,5\}$ we apply at most two bipartite complementations and obtain at most two curious graphs.
Finally, we apply at most five bipartite complementations between non-empty sets~$V_i^0$ and~$V_{i+1}^0$, yielding at most one curious graph.
This gives a total of $5+4+435+(5\times 2)+5=459$ bipartite complementations, $5+(5\times 512)+5=2570$ vertex deletions yielding at most $1+2+(5\times 1)+(5 \times 2)+1=19$ curious graphs and at most one $3$-uniform graph.
This completes the proof.\qed
\end{proof}

We can now prove our main result.
Recall that it was already known~\cite{DDP15} that the class of $(K_3,P_1+\nobreak P_5)$-free graphs has bounded clique-width, but that it was not known that this class is well-quasi-ordered.

\begin{theorem}\label{thm:main}
For $H \in \{P_2+\nobreak P_4,P_1+\nobreak P_5\}$ the class of $(K_3,H)$-free graphs is well-quasi-ordered by the labelled induced subgraph relation and has bounded clique-width.
\end{theorem}

\begin{proof}
Let $H \in \{P_2+\nobreak P_4,P_1+\nobreak P_5\}$.
By Lemmas~\ref{lem:prime-cw} and~\ref{lem:prime-wqo}, we need only consider prime graphs in this class.
Recall that a prime graph on at least three vertices cannot contain two vertices that are false twins, otherwise these two vertices would form a non-trivial module.
Therefore, by Lemma~\ref{lem:noC5-partial}, and since $H \subseteq_i S_{1,2,3}$, the classes of prime $(K_3,H)$-free graphs containing an induced~$C_7$ is precisely the graph~$C_7$.
We may therefore restrict ourselves to~$C_7$-free graphs.

Since the graphs in the class are $H$-free, it follows they contain no induced cycles on eight or more vertices.
We may therefore restrict ourselves to prime $(K_3,C_7,H)$-free graphs that either contain an induced~$C_5$ or are bipartite.
By Lemmas~\ref{lem:K3P5+P1C5} or~\ref{lem:K3P4+P2C5}, given any prime $(K_3,C_7,H)$-free that contains an induced~$C_5$, we can apply at most a constant number of vertex deletions and bipartite complementation operations to obtain a graph that is a disjoint union of $(K_3,H)$-free curious graphs and (in the $H=P_2+\nobreak P_4$ case) $3$-uniform graphs.
By Lemmas~\ref{lem:lwqo-operations}, \ref{lem:uniform-wqo} and~\ref{lem:uniform-cw}, Facts~\ref{fact:del-vert} and~\ref{fact:bip}, and Theorem~\ref{thm:curious}, it is sufficient to only consider bipartite $(K_3,C_7,H)$-free graphs.
These graphs are $H$-free bipartite graphs.
Furthermore, they form a subclass of the class of $(P_7,S_{1,2,3})$-free bipartite graphs, since $H \subseteq_i P_7, S_{1,2,3}$. 
The class of $(P_7,S_{1,2,3})$-free bipartite graphs is well-quasi-ordered by the labelled induced subgraph relation by Lemma~\ref{lem:P7S123-free-bip-wqo} and has bounded clique-width by Theorem~\ref{t-bipartite}.
This completes the proof.\qed
\end{proof}

\section{State-of-the-Art Summaries for Bigenic Graph Classes}\label{s-con}

The class of $(\overline{P_1+P_4},P_2+\nobreak P_3)$-free graphs is the only bigenic graph class left for which Conjecture~\ref{c-f} still needs to be verified; see~\cite{DLP16} for details of this claim (which can also be deduced from Theorems~\ref{t-wqowqo} 
and~\ref{thm:classification2} below).

\begin{oproblem}\label{o-con}
Is Conjecture~\ref{c-f} true for the class of $(H_1,H_2)$-free graphs~when
$H_1=\overline{P_1+P_4}$ and $H_2=P_2+\nobreak P_3$?
\end{oproblem}
The class of $(\overline{P_1+P_4},P_2+\nobreak P_3)$-free graphs is (up to an equivalence relation\footnote{Given four graphs $H_1,H_2,H_3,H_4$, the classes of $(H_1,H_2)$-free graphs and $(H_3,H_4)$-free graphs are said to be equivalent if the unordered pair $H_3,H_4$ can be obtained from the unordered pair $H_1,H_2$ by some combination of the operations (i) complementing both graphs in the pair and (ii) if one of the graphs in the pair 
is~$3P_1$, replacing it with $P_1+\nobreak P_3$ or vice versa.
If two classes are equivalent, then one of them is well-quasi-ordered by the induced subgraph relation if and only if the other one is~\cite{KL11}.
Similarly, if two classes are equivalent, then one of them has bounded clique-width if and only if the other one does~\cite{DP15}.}) one of the six remaining bigenic graph classes for which well-quasi-orderability is still open and one of the six bigenic graph classes for which boundedness of clique-width is still open.
We refer to~\cite{DLP16} for details of the first claim and to~\cite{BDJLPZ} for details of the second claim.
To make our paper self-contained we recall two theorems from these papers, which sum up our current knowledge on bigenic graph classes (including the new results of this paper).
We also include the lists of corresponding open cases below. 

Here is the state-of-the-art summary for well-quasi-orderability for bigenic graph classes.

\begin{theorem}\label{t-wqowqo}
Let~${\cal G}$ be a class of graphs defined by two forbidden induced subgraphs. Then:
\begin{enumerate}
\item ${\cal G}$ is well-quasi-ordered by the {\em labelled} induced subgraph relation if it is equivalent
to a class of $(H_1,H_2)$-free graphs such that one of the following holds:
\begin{enumerate}[(i)]
\item $H_1$ or $H_2 \ssi P_4$;
\item $H_1=sP_1$ and $H_2=K_t$ for some $s,t\geq 1$;
\item $H_1 \ssi P_1+P_3$ and $\overline{H_2} \ssi P_1+\nobreak P_5, P_2+\nobreak P_4$ or~$P_6$;
\item $H_1 \ssi 2P_1+P_2$ and $\overline{H_2} \ssi P_2+P_3$ or~$P_5$.\\[-9pt]
\end{enumerate}
\newpage
\item ${\cal G}$ is not well-quasi-ordered by the induced subgraph relation if it is equivalent to a class of $(H_1,H_2)$-free graphs such that one of the following holds: 
\begin{enumerate}[(i)]
\item neither~$H_1$ nor~$H_2$ is a linear forest;
\item $H_1 \si 3P_1$ and $\overline{H_2} \si 3P_1+\nobreak P_2, 3P_2$ or~$2P_3$;
\item $H_1 \si 2P_2$ and $\overline{H_2} \si 4P_1$ or~$2P_2$;
\item $H_1 \si 2P_1+\nobreak P_2$ and $\overline{H_2} \si 4P_1, P_2+\nobreak P_4$ or~$P_6$;
\item $H_1 \si P_1+\nobreak P_4$ and $\overline{H_2} \si P_1+\nobreak 2P_2$.
\end{enumerate}
\end{enumerate}
\end{theorem}
Theorem~\ref{t-wqowqo} does not cover six cases, which are all still open.

\begin{oproblem}\label{o-wqo}
Is the class of $(H_1,H_2)$-free graphs well-quasi-ordered 
by the induced subgraph relation 
when:
\begin{enumerate}[(i)]
\item $H_1=2P_1+\nobreak P_2$ and $\overline{H_2} \in \{P_1+\nobreak 2P_2, P_1+\nobreak P_4\}$;
\item $H_1=P_1+\nobreak P_4$ and $\overline{H_2} \in \{P_1+\nobreak P_4, 2P_2, P_2+\nobreak P_3, P_5\}$.
\end{enumerate}
\end{oproblem}
Here is the state-of-the-art summary of the boundedness of clique-width for bigenic graph classes.

\begin{theorem}\label{thm:classification2}
Let~${\cal G}$ be a class of graphs defined by two forbidden induced subgraphs.
Then:
\begin{enumerate}
\item ${\cal G}$ has bounded clique-width if it is equivalent
to a class of $(H_1,H_2)$-free graphs such that one of the following holds:\\
\begin{enumerate}[(i)]
\item \label{thm:classification2:bdd:P4} $H_1$ or $H_2 \ssi P_4$;
\item \label{thm:classification2:bdd:ramsey} $H_1=sP_1$ and $H_2=K_t$ for some $s,t\geq 1$;
\item \label{thm:classification2:bdd:P_1+P_3} $H_1 \ssi P_1+\nobreak P_3$ and $\overline{H_2} \ssi K_{1,3}+\nobreak 3P_1,\; K_{1,3}+\nobreak P_2,\;\allowbreak P_1+\nobreak P_2+\nobreak P_3,\;\allowbreak P_1+\nobreak P_5,\;\allowbreak P_1+\nobreak S_{1,1,2},\;\allowbreak P_2+\nobreak P_4,\;\allowbreak P_6,\; \allowbreak S_{1,1,3}$ or~$S_{1,2,2}$;
\item \label{thm:classification2:bdd:2P_1+P_2} $H_1 \ssi 2P_1+\nobreak P_2$ and $\overline{H_2}\ssi P_1+\nobreak 2P_2,\; 3P_1+\nobreak P_2$ or~$P_2+\nobreak P_3$;
\item \label{thm:classification2:bdd:P_1+P_4} $H_1 \subseteq_i P_1+\nobreak P_4$ and $\overline{H_2} \ssi P_1+\nobreak P_4$ or~$P_5$;
\item \label{thm:classification2:bdd:K_13} $H_1,\overline{H_2} \ssi K_{1,3}$;
\item \label{thm:classification2:bdd:2P1_P3} $H_1,\overline{H_2} \ssi 2P_1+\nobreak P_3$.\\
\end{enumerate}
\item ${\cal G}$ has unbounded clique-width if it is equivalent to a class of $(H_1,H_2)$-free graphs such that one of the following holds:
\begin{enumerate}[(i)]
\item \label{thm:classification2:unbdd:not-in-S} $H_1\not\in {\cal S}$ and $H_2 \not \in {\cal S}$;
\item \label{thm:classification2:unbdd:not-in-co-S} $\overline{H_1}\notin {\cal S}$ and $\overline{H_2} \not \in {\cal S}$;
\item \label{thm:classification2:unbdd:K_13or2P_2} $H_1 \si K_{1,3}$ or~$2P_2$ and $\overline{H_2} \si 4P_1$ or~$2P_2$;
\item \label{thm:classification2:unbdd:2P_1+P_2} $H_1 \si 2P_1+\nobreak P_2$ and $\overline{H_2} \si K_{1,3},\; 5P_1,\; P_2+\nobreak P_4$ or~$P_6$;
\item \label{thm:classification2:unbdd:3P_1} $H_1 \si 3P_1$ and $\overline{H_2} \si 2P_1+\nobreak 2P_2,\; 2P_1+\nobreak P_4,\; 4P_1+\nobreak P_2,\; 3P_2$ or~$2P_3$;
\item \label{thm:classification2:unbdd:4P_1} $H_1 \si 4P_1$ and $\overline{H_2} \si P_1 +\nobreak P_4$ or~$3P_1+\nobreak P_2$.
\end{enumerate}
\end{enumerate}
\end{theorem}
Theorem~\ref{thm:classification2} does not cover six cases, which are all still open.

\begin{oproblem}\label{oprob:twographs}
Does the class of $(H_1,H_2)$-free graphs have bounded or unbounded clique-width when:
\begin{enumerate}[(i)]
\item \label{oprob:twographs:3P_1} $H_1=3P_1$ and $\overline{H_2} \in \{P_1+\nobreak S_{1,1,3},
\allowbreak S_{1,2,3}\}$;
\item\label{oprob:twographs:2P_1+P_2} $H_1=2P_1+\nobreak P_2$ and $\overline{H_2} \in \{P_1+\nobreak P_2+\nobreak P_3,\allowbreak P_1+\nobreak P_5\}$;
\item \label{oprob:twographs:P_1+P_4} $H_1=P_1+\nobreak P_4$ and $\overline{H_2} \in \{P_1+\nobreak 2P_2,\allowbreak P_2+\nobreak P_3\}$.
\end{enumerate}
\end{oproblem}

\bibliography{mybib}
\end{document}